\documentclass[12pt]{amsart}
\usepackage{amsmath,amssymb,amsthm,amsfonts,mathrsfs}
\usepackage{amscd}
\usepackage{hyperref}
\usepackage{graphicx,subfigure}
\hypersetup{colorlinks=true,citecolor=green,linkcolor=red,urlcolor=blue,pdfstartview=FitH,
pdfauthor=Koziol/Xu,pdftitle=Hecke Modules and Supersingular Representations of U(2 1)}
\usepackage[divide={2.45cm,*,2.45cm}]{geometry}
\usepackage{enumerate}
\usepackage{color}
\usepackage[all]{xy}

\allowdisplaybreaks

\newtheorem{thm}{Theorem}[section]
\newtheorem{cor}[thm]{Corollary}
\newtheorem{lemma}[thm]{Lemma}
\newtheorem{prop}[thm]{Proposition}

\newtheorem{ass}[thm]{Assumption}
\newtheorem*{thmi}{Theorem (Corollary \ref{Karol*})}
\newtheorem*{thmn}{Theorem (Theorems \ref{ssingsl2} and \ref{SL2})}

\theoremstyle{remark}
\newtheorem*{remark}{Remark}

\theoremstyle{definition}
\newtheorem{defn}[thm]{Definition}

\newcommand{\hh}{\mathcal{H}}
\newcommand{\ol}[1]{\overline{#1}}
\newcommand{\oo}{\mathfrak{o}}
\newcommand{\pp}{\mathfrak{p}}
\newcommand{\ch}{\textnormal{char}}
\newcommand{\mT}{\mathcal{T}}
\newcommand{\mS}{\mathcal{S}}
\newcommand{\sub}[2]{\genfrac{}{}{0pt}{}{#1}{#2}}
\newcommand{\bb}{\mathbb{B}}
\newcommand{\uu}{\mathbb{U}}
\newcommand{\tj}{\mathbf{J}}
\newcommand{\T}{\textnormal{T}}

\begin{document}
\nocite{*}

\title{Hecke Modules and Supersingular Representations of U(2,1)}
\date{\today}
\author{Karol Kozio\l}
\address{Department of Mathematics, University of Toronto, Toronto, ON M5S 2E4, Canada} \email{karol@math.toronto.edu}
\author{Peng Xu}
\address{Mathematics Institute, University of Warwick, Coventry CV4 7AL, UK} 
\email{Peng.Xu@warwick.ac.uk}

\begin{abstract}
  Let $F$ be a nonarchimedean local field of odd residual characteristic $p$.  We classify finite-dimensional simple right modules for the pro-$p$-Iwahori-Hecke algebra $\mathcal{H}_C(G,I(1))$, where $G$ is the unramified unitary group $\textrm{U}(2,1)(E/F)$ in three variables.  Using this description when $C = \overline{\mathbb{F}}_p$, we define supersingular Hecke modules and show that the functor of $I(1)$-invariants induces a bijection between irreducible nonsupersingular mod-$p$ representations of $G$ and nonsupersingular simple right $\mathcal{H}_C(G,I(1))$-modules.  We then use an argument of Pa\v{s}k\={u}nas to construct supersingular representations of $G$.  
\end{abstract}

\maketitle

\section{Introduction}

This article is set in the framework of the mod-$p$ representation theory of $p$-adic reductive groups.  Our motivation comes from the possibility of a mod-$p$ Local Langlands Correspondence, that is to say a matching between (packets of) smooth mod-$p$ representations of a $p$-adic reductive group and certain Galois representations.  The case of $\textrm{GL}_2(\mathbb{Q}_p)$ has been most extensively studied, and a (semisimple) mod-$p$ Local Langlands Correspondence has been established by Breuil (\cite{Br03}) based on the explicit determination of the irreducible smooth mod-$p$ representations of $\textrm{GL}_2(\mathbb{Q}_p)$.  Moreover, this correspondence is compatible with the $p$-adic Local Langlands Correspondence established by the work of several mathematicians: see \cite{Br04}, \cite{Br10}, \cite{Co10}, \cite{Em10}, \cite{Ki09}, \cite{Ki10}, \cite{Pas10}, and the references therein.  The case of $\textrm{GL}_2(F)$ with $F\neq \mathbb{Q}_p$ is already much more complicated, however (see \cite{BP}).

In the present article, we investigate the smooth mod-$p$ representations of the unitary group $G = \textrm{U}(2,1)(E/F)$ , where $E/F$ is an unramified quadratic extension of nonarchimedean local fields of residual characteristic $p$.  The irreducible subquotients of parabolically induced representations have been classified by Abdellatif (\cite{Ab14a}).  We are interested in the smooth irreducible representations that do not appear in this fashion, which we call \emph{supersingular} representations (we will comment on this terminology at the end of this introduction).  These representations are the ones which are expected to play a crucial role in a potential Local Langlands Correspondence.  We now describe the ingredients in our method for constructing such representations, inspired by the work of Vign\'{e}ras and Pa\v{s}k\={u}nas.

Let $I(1)$ be the unique pro-$p$-Sylow subgroup of the standard Iwahori subgroup $I$ of $G$, and let $C$ denote an algebraically closed field.  The \emph{pro-$p$-Iwahori-Hecke algebra} $\hh_C(G,I(1))$ is the convolution algebra of compactly supported, $C$-valued functions on the double coset space $I(1)\backslash G/I(1)$.  Under a mild assumption on the characteristic of $C$, we determine explicitly the structure of the algebra $\hh_C(G,I(1))$ and describe its center.  This allows us to classify all simple finite-dimensional right modules of $\hh_C(G,I(1))$ for any field $C$ satisfying Assumption \ref{ass} (see Section \ref{halgs}).  We remark that some of the results of this section have been subsumed under recent work of Vign\'{e}ras (\cite{Vig14a}, \cite{Vig14b}, \cite{Vig14c}) and Abe (\cite{Abe14}).

The motivation for considering modules of the algebra $\hh_C(G,I(1))$ comes from the following observation.  Attaching to a smooth representation $\pi$ of $G$ its space of $I(1)$-invariants $\pi^{I(1)}$ yields a functor with values in the category of $\hh_C(G,I(1))$-modules.  If $C$ is of characteristic $p$, then $\pi^{I(1)}$ is nonzero provided $\pi$ is nonzero; this suggests that the functor of $I(1)$-invariants is likely to give information about representations generated by their $I(1)$-invariants (though in general, we don't expect an equivalence of categories (cf. \cite{Oll09})).

Using our explicit description of finite-dimensional simple $\hh_{\ol{\mathbb{F}}_p}(G,I(1))$ modules, we establish a bijection between irreducible smooth nonsupersingular representations of $G$ and certain simple modules (Corollary \ref{maincor}).  In particular, we show that the simple $\hh_{\ol{\mathbb{F}}_p}(G,I(1))$-modules not arising in this fashion are precisely those with a ``zero'' central character, and we call these modules \emph{supersingular}.  Our goal is to attach an irreducible smooth supersingular representation of $G$ to every supersingular $\hh_{\ol{\mathbb{F}}_p}(G,I(1))$-module.  The tool we will use is (homological) coefficient systems on the semisimple Bruhat-Tits building $X$ of $G$.

In \cite{SS97}, Schneider and Stuhler introduced coefficient systems on the Bruhat-Tits building and used them to study complex representations of $p$-adic reductive groups.  Coefficient systems were later used in the mod-$p$ setting by Pa\v{s}k\={u}nas to construct supersingular representations of $\textrm{GL}_2(F)$.  The use of coefficient systems in this context has proved extremely useful (cf. \cite{Pas04}), but so far has only been considered for the group $\textrm{GL}_2(F)$.  We adapt this method to representations of $\textrm{U}(2,1)(E/F)$, and define an analog of Pa\v{s}k\={u}nas' category of diagrams, which is easier to handle than (but equivalent to) the category of coefficient systems.

Next, we attach to every supersingular module $M$ a diagram ${D}_M$.  The $0$-homology of the corresponding coefficient system $\mathcal{D}_M$ is naturally a smooth $G$-representation, and we show that any of its nonzero irreducible quotients contain an $\hh_{\ol{\mathbb{F}}_p}(G,I(1))$-module isomorphic to $M$.  This implies that any such quotient is a supersingular representation of $G$ (Corollary \ref{maincor}).  

To produce such quotients, we next specialize to the case when the residue field of $F$ has size $p$.  In this setting we are able to construct explicitly an auxiliary coefficient system $\mathcal{E}_M$ (built out of injective envelopes of representations of finite reductive groups $\Gamma$ and $\Gamma'$ attached to $G$) along with a morphism $\mathcal{D}_M\rightarrow \mathcal{E}_M$.  This morphism induces a map on the $0$-homology of the coefficient systems, and we consider the representation afforded by the image $$\pi_{\mathcal{E}_M} = \textrm{im}(H_0(X,\mathcal{D}_M)\rightarrow H_0(X,\mathcal{E}_M)~).$$  The result here is the following:

\begin{thmi}
 Assume the residue field of $F$ has size $p$.  The representation $\pi_{\mathcal{E}_M}$ is nonzero, irreducible, admissible, and supersingular.  For nonisomorphic supersingular $\hh_{\ol{\mathbb{F}}_p}(G,I(1))$-modules $M,M'$, the representations $\pi_{\mathcal{E}_M},~\pi_{\mathcal{E}_{M'}}$ are nonisomorphic.
\end{thmi}
We remark that while $\mathcal{D}_M$ is uniquely determined, the choice of the coefficient system $\mathcal{E}_M$ is in general not unique.  Therefore, to every supersingular module $M$ we attach at least one supersingular representation; in this way, we construct at least $p^2(p+1)$ supersingular representations of $G$.

We next address the shortcomings of our method when the residue field of $F$ has size greater than $p$.  As mentioned before, our method relies on the comparison of injective envelopes for representations of the finite groups $\Gamma$ and $\Gamma'$.  When the residue field of $F$ is larger than $\mathbb{F}_p$, we demonstrate cases where the construction of Section \ref{init} would produce a coefficient system $\mathcal{E}_M$ which is ``too big,'' in the sense that we cannot guarantee irreducibility of the resulting representation.  Our main tool will be Dordowsky's Diplomarbeit (\cite{Do88}), in which the dimensions of injective envelopes of representations of $\Gamma$ are computed.

To conclude, we draw some comparisons between our results and the analogous results for the group $\textrm{SL}_2(F)$, based on results of Abdellatif in \cite{Ab13}.  The action of $\textrm{SL}_2(F)$ on its Bruhat-Tits tree $X_S$ partitions the set of vertices into two disjoint orbits and acts transitively on the edges, and therefore the results of Section \ref{diagsandcoeffs} (for $\textnormal{U}(2,1)(E/F)$) carry over formally to $\textrm{SL}_2(F)$.  When the residue field of $F$ is $\mathbb{F}_p$, we attach to every supersingular $\hh_{\ol{\mathbb{F}}_p}(\textrm{SL}_2(F),I_S(1))$-module $M_S$ two coefficient systems $\mathcal{D}_{M_S}$ and $\mathcal{E}_{M_S}$.  There is one striking difference between this case and the case of $\textrm{U}(2,1)(E/F)$, however: when the residue field of $F$ has size $p$, there is a natural choice of auxiliary diagram $\mathcal{E}_{M_S}$.  

\begin{thmn}
 Assume the residue field of $F$ has size $p$.  For each of the $p$ nonisomorphic supersingular $\hh_{\ol{\mathbb{F}}_p}(\textnormal{SL}_2(F),I_S(1))$-modules $M_S$ there is a pair of associated coefficient systems $(\mathcal{D}_{M_S},\mathcal{E}_{M_S})$.  The resulting $\textnormal{SL}_2(F)$-representation afforded by $$\pi_{{M_S}} = \textnormal{im}(H_0(X_S,\mathcal{D}_{M_S})\rightarrow H_0(X_S,\mathcal{E}_{M_S})~)$$ is nonzero, irreducible, admissible, and supersingular.  For nonisomorphic $\hh_{\ol{\mathbb{F}}_p}(\textnormal{SL}_2(F),I_S(1))$-modules $M_S,M_S'$, the representations $\pi_{{M_S}},~\pi_{{M_S'}}$ are nonisomorphic.  In particular, when $F = \mathbb{Q}_p,$ we recover in this way all $p$ nonisomorphic supersingular representations of $\textnormal{SL}_2(\mathbb{Q}_p)$ as classified in \cite{Ab13}.  
\end{thmn}

\noindent\textbf{Remark on Terminology.}  We briefly address our choice of nomenclature.  The notion of supersingularity was introduced by Barthel and Livn\'{e} (\cite{BL94} and \cite{BL95}) in their classification of smooth,  irreducible, nonsupercuspidal mod-$p$ representations of $\textrm{GL}_2(F)$.  For a general connected reductive group, an irreducible smooth representation $\pi$ is called \emph{supersingular} if the localization of a certain Hom-space associated to $\pi$ is trivial, while $\pi$ called \emph{supercuspidal} if it is not a subquotient of a parabolically induced representation.  Forthcoming work of Abe--Henniart--Herzig--Vign\'{e}ras (\cite{AHHV}) establishes the equivalence of these two notions.  In anticipation of these results, we will make no distinction between supercuspidal and supersingular representations, and henceforth only use the term supersingular.

\noindent\textbf{Acknowledgements.}   The authors would like to thank their advisors, Professors Rachel Ollivier and Shaun Stevens, for suggesting this problem, as well as for many helpful and illuminating discussions throughout the course of working on this paper.  Part of this work was done while the first author was a visitor at the University of East Anglia, and he wishes to thank the institution for their support.  Additionally, the first author would like to warmly thank Professor James Humphreys for several enlightening discussions, and for providing a copy of Dordowsky's thesis.  The first author was supported by NSF Grant DMS-0739400; the second author was supported by EPSRC Grant EP/H00534X/1.

\section{Notation}

\subsection{General Notation}

Fix an odd prime number $p$.  Let $F$ be a nonarchimedean local field of residual characteristic $p$, with ring of integers $\mathfrak{o}_F$ and maximal ideal $\mathfrak{p}_F$.  Fix a uniformizer $\varpi$ and let $k_F = \mathfrak{o}_F/\mathfrak{p}_F$ denote the residue field of size $q = p^f$.  We fix also a separable closure $\overline{F}$ of $F$, and let $k_{\ol{F}}$ denote its residue field. 

Let $E$ denote the unique unramified extension of degree 2 in $\ol{F}$.  We denote by $\oo_E, \pp_E$, etc., the analogous objects for $E$.  Since $E$ is unramified, we may and do take $\varpi$ as our uniformizer.  Let $\iota:k_{\ol{F}}\stackrel{\sim}{\rightarrow} \ol{\mathbb{F}}_p$ denote a fixed isomorphism, and assume that every $\ol{\mathbb{F}}_p^\times$-valued character factors through $\iota$.  We identify $k_F$ and $k_E$ with $\mathbb{F}_q$ and $\mathbb{F}_{q^2}$, respectively, using the isomorphism $\iota$.  We will also identify $\mathbb{F}_{q^2}^\times$ with the image of the Teichm\"{u}ller lifting map $[\ \cdot\ ]:\mathbb{F}_{q^2}^\times\rightarrow \mathfrak{o}_E^\times$ when convenient.  

We let $x\mapsto\ol{x}$ denote the nontrivial Galois automorphism of $E$ fixing $F$ (which induces the automorphism $x\mapsto x^q$ on $\mathbb{F}_{q^2}$).  We shall write $E = F(\sqrt{\epsilon})$, where $\epsilon\in \mathfrak{o}_F^\times$ is some fixed but arbitrary nonsquare unit, so that $\ol{\sqrt{\epsilon}} = - \sqrt{\epsilon}$.  We define $\textnormal{U}(1)(E/F)$ the kernel of the norm map
\begin{center}
 \begin{tabular}{rl}
  $\textnormal{N}_{E/F}:$ & $E^\times\rightarrow F^\times$\\
 & $x\mapsto x\ol{x}$.
 \end{tabular}
\end{center}

Denote by $G$ the $F$-rational points of the algebraic group $\mathbf{U}(2,1)$, defined and quasisplit over $F$.  Explicitly, we take $G$ to have the form 
$$G = \left\{g\in \textrm{GL}_3(E): \ol{g}^\top\begin{pmatrix}0 & 0 & 1\\ 0 & 1 & 0\\ 1 & 0 & 0\end{pmatrix}g = \begin{pmatrix}0 & 0 & 1\\ 0 & 1 & 0\\ 1 & 0 & 0\end{pmatrix}\right\}.$$

The group $G$ possesses, up to conjugacy, two maximal compact subgroups (cf. \cite{Ti79}, Sections 2.10 and 3.2), given by $$K := \textrm{GL}_3(\mathfrak{o}_E)\cap G\quad \textrm{and}\quad K' := \begin{pmatrix}\mathfrak{o}_E & \mathfrak{o}_E & \mathfrak{p}_E^{-1}\\ \mathfrak{p}_E & \mathfrak{o}_E & \mathfrak{o}_E \\ \mathfrak{p}_E & \mathfrak{p}_E & \mathfrak{o}_E \end{pmatrix} \cap G.$$  
Let $K_1, K'_1$ be the following subgroups of $G$:
$$K_1 :=\begin{pmatrix}1 + \mathfrak{p}_E & \mathfrak{p}_E & \mathfrak{p}_E\\ \mathfrak{p}_E & 1 + \mathfrak{p}_E & \mathfrak{p}_E \\ \mathfrak{p}_E & \mathfrak{p}_E & 1 + \mathfrak{p}_E\end{pmatrix}\cap G,\qquad K'_1 := \begin{pmatrix}1 + \mathfrak{p}_E & \mathfrak{o}_E & \mathfrak{o}_E \\ \mathfrak{p}_E & 1 + \mathfrak{p}_E & \mathfrak{o}_E \\ \mathfrak{p}_E^2 & \mathfrak{p}_E & 1 + \mathfrak{p}_E\end{pmatrix}\cap G.$$
The group $K_1$ (resp. $K'_1$) is the maximal normal pro-$p$ subgroup of $K$ (resp. $K'$).  We define 
$$\Gamma := K/K_1 \cong \textnormal{U}(2,1)(\mathbb{F}_{q^2}/\mathbb{F}_{q}),\qquad \Gamma' := K'/K'_1 \cong (\textnormal{U}(1,1)\times\textnormal{U}(1))(\mathbb{F}_{q^2}/\mathbb{F}_{q}),$$
where $\textnormal{U}(1,1)(\mathbb{F}_{q^2}/\mathbb{F}_q)$ denotes the unitary group defined with respect to the matrix $\left(\begin{smallmatrix}0 & 1 \\ 1 & 0\end{smallmatrix}\right)$.  

We let $\bb$ denote the upper-triangular Borel subgroup of $\Gamma$ and $\uu$ its unipotent radical; likewise, let $\bb'$ denote the \emph{lower}-triangular Borel subgroup of $\Gamma'$ and $\uu'$ its unipotent radical.  The groups $\uu$ and $\uu'$ are $p$-Sylow subgroups of $\Gamma$ and $\Gamma'$, respectively.  We define the Iwahori subgroup as $I := K \cap K'$, which we may also think of as the preimage in $K$ under the reduction-modulo-$\varpi$ map of $\mathbb{B}$.  We denote by $I(1)$ the unique pro-$p$-Sylow subgroup of $I$, which is the preimage of $\uu$.  

Let $U$ and $U^-$ denote the upper- and lower-triangular unipotent elements of $G$, respectively, and define
$$u(x,y):=\begin{pmatrix}1 & x & y \\ 0 & 1 & -\ol{x} \\ 0 & 0 & 1\end{pmatrix},\qquad u^-(x,y):=\begin{pmatrix}1 & 0 & 0\\ x & 1 & 0\\ y & -\ol{x} & 1\end{pmatrix},$$
where $x,y\in E$ satisfy $x\ol{x} + y + \ol{y} = 0$.  We have $u(x,y)^{-1} = u(-x,\ol{y}), u^-(x,y)^{-1} = u^-(-x,\ol{y})$.   

We define the following distinguished elements of $G$:
\begin{center}
\begin{tabular}{rc}
$\displaystyle{s:=\begin{pmatrix}0 & 0 & 1\\ 0 & 1 & 0\\ 1 & 0 & 0\end{pmatrix}},$ &  $\displaystyle{s':=\begin{pmatrix} 0 & 0 & \varpi^{-1}\\0 & 1 & 0\\ \varpi & 0 & 0\end{pmatrix}},$\\
$\displaystyle{n_s:=\begin{pmatrix}0 & 0 & -\sqrt{\epsilon}^{-1}\\ 0 & 1 & 0\\ \sqrt{\epsilon}& 0 & 0\end{pmatrix}},$ &  $\displaystyle{n_{s'}:=\begin{pmatrix}0 & 0 & -\varpi^{-1}\sqrt{\epsilon}^{-1}\\ 0 & 1 & 0\\ \varpi\sqrt{\epsilon} & 0 & 0\end{pmatrix}},$\\
$\displaystyle{\alpha := s's = \begin{pmatrix}\varpi^{-1} & 0 & 0\\ 0 & 1 & 0\\ 0 & 0 & \varpi\end{pmatrix}},$ & $\displaystyle{\alpha^{-1}:= ss' = \begin{pmatrix}\varpi & 0 & 0\\ 0 & 1 & 0\\ 0 & 0 & \varpi^{-1}\end{pmatrix}}.$
\end{tabular}
\end{center}

\subsection{Weyl Groups}  The diagonal maximal torus $T$ of $G$ consists of all elements of the form $$\begin{pmatrix}a & 0 & 0 \\ 0 & \delta & 0 \\ 0 & 0 & \ol{a}^{-1}\end{pmatrix},$$ with $a\in E^\times, \delta\in \textnormal{U}(1)(E/F)$.  Note that $T$ is not split over $F$.  Let 
$$T_0 := T\cap K = T\cap K',\quad T_1:= T\cap K_1 = T\cap K'_1,$$ $$H := T_0/T_1 \cong I/I(1) \cong \mathbb{F}_{q^2}^\times\times \textnormal{U}(1)(\mathbb{F}_{q^2}/\mathbb{F}_q).$$  
We will identify the characters of $H$ and those of $I/I(1)$.

Let $N$ denote the normalizer of $T$ in $G$, and define the finite and affine Weyl group, respectively, as
$$W := N/T\qquad\textnormal{and}\qquad W_{\textnormal{aff}} := N/T_0.$$
The group $W_{\textrm{aff}}$ is a Coxeter group, generated by the classes of the two reflections $s$ and $s'$.  We have a decomposition $G = INI$, where two cosets $InI$ and $In'I$ are equal if and only if $n$ and $n'$ have the same image in $W_{\textrm{aff}}$.  This yields the Bruhat decomposition for $G$: 
$$G = \bigsqcup_{w\in W_{\textrm{aff}}}IwI;$$ 
here we engage in the standard abuse of notation, letting $IwI$ denote $I\dot{w}I$ for any preimage $\dot{w}\in N$ of $w\in W_{\textnormal{aff}}$.  We will take as our double coset representatives the elements $\alpha^n, n_s\alpha^n$, for $n\in\mathbb{Z}$.  We let $\ell$ denote the length of an element of $W_{\textrm{aff}}$, defined by $$q^{\ell(w)} = [IwI:I]$$ (cf. Section 3.3.1 in \cite{Ti79}).  In particular, we have $\ell(n_s) = 3,\ \ell(n_{s'}) = 1$.  Moreover, one easily checks the pair $(I,N)$ forms a BN pair (cf. \cite{AB08}, Definition 6.55).

\section{Hecke Algebras}\label{halgs}

\subsection{Pro-$p$-Iwahori-Hecke Algebra}  We let $C$ denote an algebraically closed field, and $\mathcal{REP}_C(G)$ the category of smooth representations of $G$.  Given a smooth $C$-representation $\sigma$ of a closed subgroup $J$ of $G$, we denote by $\textnormal{ind}_J^G(\sigma)$ (resp. $\textnormal{c-ind}_J^G(\sigma)$) the induction (resp. compact induction) of $\sigma$ from $J$ to $G$. 

Let $\pi$ be a smooth $C$-representation of $G$.  Frobenius Reciprocity gives 
$$\pi^{I(1)}\cong \textrm{Hom}_{I(1)}(1,\pi|_{I(1)})\cong \textrm{Hom}_G(\textrm{c-ind}_{I(1)}^G(1),\pi),$$ 
where $1$ denotes the trivial character of $I(1)$.  The \emph{pro-$p$-Iwahori-Hecke algebra} 
$$\hh_C(G,I(1)) := \textnormal{End}_G(\textnormal{c-ind}_{I(1)}^G(1))$$ 
is the algebra of $G$-equivariant endomorphisms of the universal module $\textnormal{c-ind}_{I(1)}^G(1)$.  This algebra has a natural right action on $\textrm{Hom}_G(\textrm{c-ind}_{I(1)}^G(1),\pi)$ by pre-composition, which induces a right action on $\pi^{I(1)}$.  In this way, we obtain the functor of $I(1)$-invariants, $\pi\mapsto\pi^{I(1)}$, from the category of smooth $C$-representations of $G$ to the category of right $\hh_C(G,I(1))$-modules.  

By adjunction, we have a natural identification $$\hh_C(G,I(1)) \cong \textnormal{Hom}_{I(1)}(1,\textnormal{c-ind}_{I(1)}^G(1)|_{I(1)})\cong \textnormal{c-ind}_{I(1)}^G(1)^{I(1)},$$ so we may view endomorphisms of $\textnormal{c-ind}_{I(1)}^G(1)$ as compactly supported functions on $G$ which are $I(1)$-biinvariant.  This leads to the following definition.

\begin{defn}
 Let $g\in G$.  We let $\T_g\in \hh_C(G,I(1))$ denote the endomorphism of $\textnormal{c-ind}_{I(1)}^G(1)$ corresponding by adjunction to the characteristic function of ${I(1)gI(1)}$; in particular, $\T_g$ maps the characteristic function of $I(1)$ to the characteristic function of $I(1)gI(1)$.  
\end{defn}

Since $W_{\textrm{aff}} = N/T_0$ is a set of representatives for the double coset space $I\backslash G/I$, the group $N/T_1$ gives a set of representatives for $I(1)\backslash G/I(1)$.  We therefore only consider the operators $\T_n$, where $n$ is a representative of a coset in $N/T_1$.  These operators give a basis for $\hh_C(G,I(1))$ as a vector space over $C$.  Using the natural adjunction isomorphisms above, we see that if $\pi$ is a smooth $C$-representation of $G$, $v\in \pi^{I(1)}$, and $n\in N$ then 
\begin{equation}\label{act}
v\cdot \T_n = \sum_{u \in I(1)\backslash I(1)nI(1)}u^{-1}.v = \sum_{u\in I(1)/I(1)\cap n^{-1}I(1)n} un^{-1}.v.
\end{equation}

\subsection{Decomposition of the pro-$p$-Iwahori-Hecke Algebra}\label{decompalg}

Let $\widehat{H}$ denote the group of all $C^\times$-valued characters of $H = T_0/T_1$, and let $\chi\in\widehat{H}$.  We define $\zeta:\mathbb{F}_{q^2}^\times\rightarrow C^\times$ and $\eta:\textnormal{U}(1)(\mathbb{F}_{q^2}/\mathbb{F}_q)\rightarrow C^\times$ by  
$$\zeta(a) := \chi\begin{pmatrix}a & 0 & 0 \\ 0 & \ol{a}a^{-1} & 0 \\ 0 & 0 & \ol{a}^{-1}\end{pmatrix},\qquad \eta(\delta) := \chi\begin{pmatrix}1 & 0 & 0 \\ 0 & \delta & 0 \\ 0 & 0 & 1\end{pmatrix},$$
where $a\in \mathbb{F}_{q^2}^\times, \delta \in \textnormal{U}(1)(\mathbb{F}_{q^2}/\mathbb{F}_q)$.  We stress that the characters $\zeta$ and $\eta$ \emph{depend on $\chi$}, though we will supress this dependence from our notation, and write $\chi = \zeta\otimes\eta$ when convenient.  We denote by $\chi^s$ the character given by $\chi^s(h) := \chi(n^{-1}hn)$, where $h\in H$ and $n\in N\smallsetminus T$.  

\begin{defn}
Let $\chi\in\widehat{H}$.  We say $\chi$ is \emph{of trivial Iwahori type} if $\chi$ factors through the determinant, $\chi$ is \emph{hybrid} if $\chi^s = \chi$, but $\chi$ does not factor through the determinant, and $\chi$ is \emph{regular} if $\chi^s \neq \chi$.  
\end{defn}
Note that $\chi = \zeta\otimes\eta$ factors through the determinant if and only if $\zeta$ is trivial, and $\chi^s = \chi$ if and only if $\zeta^{q+1}$ is trivial.  For $\chi\in\widehat{H}$, we define a representation $\gamma_\chi$ of $H$ by 
$$\gamma_\chi := \begin{cases}\chi & \textnormal{if}~ \chi^s = \chi,\\ \chi\oplus\chi^s & \textnormal{if}~ \chi^s\neq \chi.\end{cases}$$  

From this point onwards, we make the following technical assumption:
\begin{ass}\label{ass}
 The integers $\ch(C)$ and $|H|$ are relatively prime.
\end{ass}

With this hypothesis, we will decompose $\hh_C(G,I(1))$ into blocks indexed by $W$-orbits of $C$-characters of $H$.  

\begin{defn}
 For a $C$-character $\chi$ of $H$, we define $$e_\chi := |H|^{-1}\sum_{h\in H}\chi(h)\T_h,$$  $$e_{\gamma_\chi} := \begin{cases}e_\chi & \textnormal{if}~\chi^s = \chi,\\ e_\chi + e_{\chi^s} & \textnormal{if}~\chi^s\neq\chi.\end{cases}$$ 
Here $\T_h\in\hh_{C}(G,I(1))$ denotes the operator $\T_{t_0}$, where $t_0\in T_0$ is a preimage of $h\in H$.  
\end{defn}

The operators $e_\chi$ have the following properties:
\begin{itemize}
 \item $e_\chi e_\chi = e_\chi$,
 \item $e_\chi e_{\chi'} = 0$ for $\chi\neq \chi'$,
 \item $\T_1 = \sum_{\chi\in \widehat{H}}e_\chi$.
 
\end{itemize}
These follow readily from the orthogonality relations of characters.  Applying these relations to $\pi^{I(1)}$ gives the following lemma.

\begin{lemma}\label{decomp}
 Let $\pi$ be a smooth $C$-representation of $G$.  Then $(\pi^{I(1)})\cdot e_\chi = \pi^{I,\chi}$, and $\pi^{I(1)}\cong \bigoplus_{\chi\in \widehat{H}} (\pi^{I(1)})\cdot e_\chi = \bigoplus_{\chi\in \widehat{H}} \pi^{I,\chi}$.  Here $\pi^{I,\chi} = \{v\in \pi : i.v = \chi(i)v\ \textrm{for every}\ i\in I\}$ is the $\chi$-isotypic subspace of $\pi$.    
\end{lemma}

\begin{proof}
 Since $I(1)$ is normal in $I$ and $I/I(1)\cong H$ is abelian and of order prime to $\ch(C)$, the action of $I$ on $\pi^{I(1)}$ is semisimple and decomposes as a sum of characters.  As (lifts of) elements of $H$ normalize $I(1)$, equation \eqref{act} implies that $(\pi^{I(1)})\cdot e_\chi= \pi^{I,\chi}$.  \end{proof}

We now use the (central) idempotents $e_{\gamma_\chi}$ to decompose the algebra $\hh_C(G,I(1))$.  Denote by $\hh_C(G,\gamma_\chi)$ the algebra $\textrm{End}_G(\textrm{c-ind}_I^G(\gamma_\chi))$.  

\begin{prop}\label{orbits}
 There is an isomorphism of $C$-algebras
$$\hh_C(G,I(1)) \cong \bigoplus_{\gamma_\chi}\hh_C(G,\gamma_\chi)\cong \bigoplus_{\gamma_\chi}\hh_C(G,I(1))e_{\gamma_\chi},$$
the sums taken over all $W$-orbits of $C$-characters of $H$.
\end{prop}

\begin{proof}
 Assumption \ref{ass} guarantees that the regular representation of $I/I(1)$ is semisimple.  Using this fact, the proof is nearly identical to that in \cite{Vig04}, Proposition 3.1.  
\end{proof}

\begin{thm}\label{propstr}\hfill
 \begin{enumerate}[(i)]
  \item The algebra $\hh_C(G,I(1))$ is generated by the elements $\T_{n_s}, \T_{n_{s'}}$ and $e_{\chi}, \chi\in\widehat{H}$, subject to the following relations:
  \begin{enumerate}[(a)]
  \item $$\T_{n_s}e_\chi = e_{\chi^s}\T_{n_s},\qquad \T_{n_{s'}}e_\chi = e_{\chi^s}\T_{n_{s'}},$$
$$e_{\chi}e_{\chi'} = \begin{cases}e_\chi & \textnormal{if}~ \chi' = \chi,\\ 0 & \textnormal{if}~ \chi'\neq \chi.\end{cases}$$
  \item If $\chi$ is of trivial type, then 
  $$\T_{n_s}^2e_\chi = (q^3 - 1)\T_{n_s}e_\chi + q^3e_\chi,\qquad \T_{n_{s'}}^2e_\chi = (q - 1)\T_{n_{s'}}e_\chi + qe_\chi.$$
 If $\chi$ is hybrid, then 
 $$\T_{n_s}^2e_\chi = (q - q^2)\T_{n_s}e_\chi + q^3e_\chi,\qquad \T_{n_{s'}}^2e_\chi = (q - 1)\T_{n_{s'}}e_\chi + qe_\chi.$$
  If $\chi$ is regular, then 
  $$\T_{n_s}^2e_\chi = \zeta(-1)q^3e_\chi,\qquad \T_{n_{s'}}^2e_\chi = \zeta(-1)qe_\chi.$$
  \end{enumerate}
  \item The center $\mathcal{Z}$ of $\hh_C(G,I(1))$ is generated by the idempotents $e_{\gamma_\chi}$, and the elements
 $\begin{cases} (\T_{n_s}(\T_{n_{s'}} - (q - 1)) + \T_{n_{s'}}(\T_{n_s} - (q^3 - 1)) + 1)e_\chi & \textit{for}~ \chi~\textit{of trivial type},\\
  (\T_{n_s}(\T_{n_{s'}} - (q - 1)) + \T_{n_{s'}}(\T_{n_s} - (q - q^2)))e_\chi & \textit{for}~ \chi~\textit{hybrid},\\
  (\T_{n_{s'}}\T_{n_s}e_\chi + \T_{n_s}\T_{n_{s'}}e_{\chi^s})~ \textit{and} & \\
  \qquad (\T_{n_{s'}}\T_{n_s}e_{\chi^s} + \T_{n_s}\T_{n_{s'}}e_\chi) & \textit{for}~ \chi~\textit{regular}. \end{cases}$
 \end{enumerate}
\end{thm}

\begin{proof}
Let $\mathcal{M}$ be the subalgebra of $\hh_C(G,I(1))$ generated by $\T_{n_s}, \T_{n_{s'}}$ and the operators $e_\chi$ for every $\chi\in \widehat{H}$.  By Proposition \ref{orbits}, $\mathcal{M}e_{\gamma_\chi}$ is a subalgebra of $\hh_C(G,\gamma_\chi)$.  Propositions \ref{triv}, \ref{str}, \ref{mod}, and \ref{mod2} now show that the elements $\T_{n_s}e_\chi, \T_{n_{s'}}e_\chi, \T_{n_s}e_{\chi^s}$ and $\T_{n_{s'}}e_{\chi^s}$ of $\mathcal{M}e_{\gamma_\chi}$ generate $\hh_C(G,\gamma_\chi)$, and therefore $\mathcal{M} = \hh_C(G,I(1))$.  Furthermore, these propositions and Corollary \ref{twistmatalg} give the quadratic relations and the structure of the center.  
\end{proof}

\begin{remark}
 Let $h_s:E^\times \rightarrow T$ be the homomorphism defined by 
 $$h_s(y) = \begin{pmatrix}y & 0 & 0\\ 0 & \ol{y}y^{-1} & 0 \\ 0 & 0 & \ol{y}^{-1}\end{pmatrix},$$ and set $$\tau_s := (q + 1)\sum_{y\in \mathbb{F}_{q^2}^\times}\T_{h_s(y)} - q\sum_{y\in\mathbb{F}_q^\times}\T_{h_s(y)},\qquad \tau_{s'} := \sum_{y\in\mathbb{F}_q^\times} \T_{h_s(y)}.$$  
 Using Fourier inversion and the theorem above, the quadratic relations take the form
\begin{eqnarray*}
 \T_{n_s}^2 & = & \T_{n_s}\tau_s + q^3\T_{h_s(-1)}\\
 \T_{n_{s'}}^2 & = & \T_{n_{s'}}\tau_{s'} + q\T_{h_s(-1)}.
\end{eqnarray*}
Moreover, we see that the center $\mathcal{Z}$ of $\hh_C(G,I(1))$ is generated by the central idempotents $e_{\gamma_\chi}$ and the elements 
$$\T_{n_{s'}}\T_{n_s}\vartheta_1 + \T_{n_s}\T_{n_{s'}}\vartheta_2 - \T_{n_s}\tau_{s'} - \T_{n_{s'}}\tau_s + (q-1)\tau_s,$$
$$\T_{n_{s'}}\T_{n_s}\vartheta_2 + \T_{n_s}\T_{n_{s'}}\vartheta_1 - \T_{n_s}\tau_{s'} - \T_{n_{s'}}\tau_s + (q-1)\tau_s.$$
Here $$\vartheta_1 := \sum_{\chi^s = \chi}e_\chi + 2\sum_{\sub{\chi^s\neq \chi}{\chi\in \{\chi,\chi^s\}}}e_\chi,\quad \vartheta_2 := \sum_{\chi^s = \chi}e_\chi + 2\sum_{\sub{\chi^s\neq \chi}{\chi^s\in \{\chi,\chi^s\}}}e_{\chi^s},$$
where the sums are taken over $W$-orbits of $C$-characters, such that $\vartheta_1 + \vartheta_2 = 2\T_1$.  
\end{remark}

In light of Theorem \ref{propstr}, we make the following definition:
\begin{defn}\label{ssingdef1}
 Assume $\ch(C) = p$, and let $M$ be a nonzero simple right $\hh_C(G,I(1))$-module which admits a central character.  We say $M$ is \emph{supersingular} if every generator of the center $\mathcal{Z}$ (as given in Theorem \ref{propstr}) which is not a central idempotent $e_{\gamma_\chi}$, acts by 0.  
\end{defn}

In the subsequent sections, we describe the structures of the Hecke algebras $\hh_C(G,\gamma_\chi)$.  From the descriptions of these blocks, we obtain Theorem \ref{propstr}, and identify the supersingular modules of $\hh_C(G,I(1))$ when $\ch(C) = p$.

\subsection{The Trivial Case}

We first assume that $\chi$ is of trivial type, meaning $\chi$ factors through the determinant and $\chi = \eta\circ\det$, for $\eta$ a character of $\textnormal{U}(1)(\mathbb{F}_{q^2}/\mathbb{F}_q)$ (which we also view as a character of $\textnormal{U}(1)(E/F)$).  

Let $\mathbf{1}_I\in \textrm{c-ind}_{I}^G(\chi)$ denote the function with support in $I$, taking the value 1 at the identity.  We let $\mT_{n_s}$ (resp. $\mT_{n_{s'}}$) denote the endomorphism of $\textrm{c-ind}_{I}^G(\chi)$ sending $\mathbf{1}_I$ to the function with support $In_sI$ (resp. $In_{s'}I$), taking the value 1 at $n_s$ (resp. $n_{s'}$), on which $I$ acts by $\chi$.  In the notation of the previous section, we have $$\mT_{n_s} = \T_{n_s}e_\chi,\qquad \mT_{n_{s'}} = \T_{n_{s'}}e_\chi.$$  

We now arrive at the following result on the structure of $\hh_C(G,\chi)$:
\begin{prop}\label{triv}
 The algebra $\hh_C(G,\chi)$ is a noncommutative algebra, generated by $\mT_{n_s}$ and $\mT_{n_{s'}}$, subject to the relations
\begin{eqnarray*}
 (\mT_{n_s} + 1)(\mT_{n_s} - q^3) & = & 0\\
 (\mT_{n_{s'}} + 1)(\mT_{n_{s'}} - q) & = & 0.
\end{eqnarray*}
The center $\mathcal{Z}_\chi$ is generated by $Z = \mT_{n_s}(\mT_{n_{s'}} - (q - 1)) + \mT_{n_{s'}}(\mT_{n_s} - (q^3 - 1)) + 1$.  We have an isomorphism of algebras $$\hh_C(G,\chi)\cong C\langle X,Y\rangle/(X^2 +(1-q^3)X - q^3, Y^2 + (1-q)Y - q),$$ sending $\mT_{n_s}$ to $X$ and $\mT_{n_{s'}}$ to $Y$.  Here $C\langle X,Y \rangle$ denotes the noncommutative polynomial algebra in two variables over $C$.  
\end{prop}

\begin{remark}
 Note that using the length function on $W_{\textrm{aff}}$, the Hecke relations take the simple form $(\mT_n + 1)(\mT_n - q^{\ell(n)}) = 0$, where $n = n_s$ or $n_{s'}$.
\end{remark}

\begin{proof}
 See the proof of Proposition \ref{str} below.  
\end{proof}

Given this result, we can quickly classify the finite-dimensional simple right $\hh_C(G,\chi)$-modules.  

\begin{defn}
\begin{enumerate}[(i)]
 \item Let $(\theta,\theta')\in \{-1,q^3\}\times\{-1,q\}$.  We define the characters $\mu_{\theta,\theta'}:\hh_C(G,\chi)\rightarrow C$ by 
$$\mT_{n_s} \mapsto \theta,\quad \mT_{n_{s'}} \mapsto \theta'.$$ 
The central element $Z$ maps to $\theta(\theta' - q + 1) + \theta'(\theta - q^3 + 1) + 1\in\{q^3 + q + 1,-q^4\}$.  

 \item Let $\langle v_1, v_2\rangle_C$ be a two-dimensional vector space over $C$, and let $\lambda\in C$.  We define $M(\lambda)$ to be the following right $\hh_C(G,\chi)$-module:
\begin{center}
\begin{tabular}{cclccl}
 $v_1\cdot \mT_{n_s}$ & = & $-v_1,$ & $v_1\cdot \mT_{n_{s'}}$ & = & $v_2$\\
 $v_2\cdot \mT_{n_s}$ & = & $(\lambda - q)v_1 + q^3 v_2,$ & $v_2\cdot \mT_{n_{s'}}$ & = & $qv_1 + (q-1)v_2$
\end{tabular}
\end{center}
The central element $Z$ acts by $\lambda$.  
\end{enumerate}
\end{defn}

One may check directly that the action of $\hh_C(G,\chi)$ on $M(\lambda)$ is well-defined.  This fact will also be made clear in the proof of Theorem \ref{thmiw}.

\begin{prop}\label{iwahorired}
 Assume $q^3 + 1\neq 0$ in $C$ (note that $q + 1\neq 0$ by Assumption \ref{ass}).  Then the module $M(\lambda)$ is reducible if and only if $\lambda = q^3 + q + 1$ or $\lambda = -q^4$.  In these cases, we have the following nonsplit exact sequences:
$$0\rightarrow \mu_{q^3,q}\rightarrow M(q^3 + q + 1)\rightarrow \mu_{-1,-1}\rightarrow 0$$
$$0\rightarrow \mu_{q^3,-1}\rightarrow M(-q^4)\rightarrow \mu_{-1.q}\rightarrow 0$$
\end{prop}

\begin{proof}
 Assume that $M(\lambda)$ is reducible, so that we have some character $\mu_{\theta,\theta'}\subset M(\lambda)$.  By examining the action of the center, we conclude that $\lambda\in\{q^3 + q + 1, -q^4\}$.  Assume $\lambda = q^3 + q + 1$; the other case is similar.  One then checks directly that
$$\langle v_1 + v_2\rangle_C\cong\mu_{q^3,q}\subset M(q^3 + q + 1)\quad \textnormal{and}\quad M(q^3 + q + 1)/\mu_{q^3,q}\cong \mu_{-1,-1}.$$  
The assumption $q^3 + 1 \neq 0$ guarantees that the surjection $M(q^3 + q + 1) \twoheadrightarrow \mu_{-1,-1}$ cannot split.
\end{proof}

\begin{prop}\label{exc}
 Assume $q^3 + 1 = 0$ in $C$.  Then the module $M(\lambda)$ is reducible if and only if $\lambda = q$.  In this case the module decomposes as $M(q) \cong \mu_{-1,q}\oplus\mu_{-1,-1}$.  
\end{prop}

\begin{proof}
Assume that $M(\lambda)$ is reducible, so that it contains either $\mu_{-1,-1}$ or $\mu_{-1,q}$.  By examining the action of the center, we conclude that $\lambda = q$.  One checks directly that 
$$M(q) \cong \langle v_1 + v_2\rangle_C \oplus \langle -qv_1 + v_2\rangle_C \cong \mu_{-1,q}\oplus\mu_{-1,-1}$$
(Assumption \ref{ass} guarantees that the sum is direct).  
\end{proof}

We now imitate the proof of Theorem 1.2 in \cite{Vig04} to classify simple right $\hh_C(G,\chi)$-modules.

\begin{thm}
\label{thmiw}
 Every finite-dimensional simple right $\hh_C(G,\chi)$-module is isomorphic to either a character $\mu_{\theta,\theta'}, (\theta,\theta')\in\{-1,q^3\}\times\{-1,q\}$, or a module of the form $M(\lambda), \lambda\neq q^3 + q + 1, -q^4$.  
\end{thm}

\begin{proof}
 Assume $M$ is a nonzero simple right module which is not a character, and assume that $Z$ acts by $\lambda$.  We claim that the space $\ker(\mT_{n_s} + 1)$ is a nontrivial proper subspace of $M$.  Indeed, if $\ker(\mT_{n_s} + 1) = \{0\}$ or $M$, the element $\mT_{n_s}$ would act by a scalar, and any nonzero eigenvector for $\mT_{n_{s'}}$ would generate a one-dimensional submodule.  This gives a contradiction, since $M$ was assumed simple of dimension greater than 1.

The element $\mT_{n_{s'}}(\mT_{n_s} - q^3)$ maps $\ker(\mT_{n_s} + 1)$ into itself, and therefore has an eigenvector $v$ in $\ker(\mT_{n_s} + 1)$.  We have 
\begin{eqnarray*}
 v\cdot\mT_{n_{s'}}(\mT_{n_s} - q^3) & = & v\cdot(Z - (\mT_{n_s} + 1)\mT_{n_{s'}} + (q - 1)\mT_{n_s} - 1)\\
 & = & \lambda v + (q - 1)v\cdot\mT_{n_s} - v\\
 & = & (\lambda - q)v.
\end{eqnarray*}
Consider now the subspace $V := \langle v\rangle_C + \langle v\cdot\mT_{n_{s'}}\rangle_C$.  The quadratic relations and the computation above show that $V$ is stable under $\hh_C(G,\chi)$, and therefore must be all of $M$ by simplicity.  Moreover, since $M$ was assumed to be of dimension greater than one, we have $v\cdot\mT_{n_{s'}}\neq 0$, and the sum $\langle v\rangle_C + \langle v\cdot\mT_{n_{s'}}\rangle_C$ is direct.  Writing out the actions of $\mT_{n_s}$ and $\mT_{n_{s'}}$ on the basis $\{v, v\cdot\mT_{n_{s'}}\}$ shows that $M\cong M(\lambda)$.  We again use simplicity of $M$ to deduce that $\lambda\neq q^3 + q + 1, -q^4$.
\end{proof}

\subsection{The Hybrid Case}

We now assume that $\chi^s = \chi = \zeta\otimes\eta$, but that $\chi$ does not factor through the determinant.  This condition implies that the character $\zeta$ is nontrivial.  In addition, we have $\zeta(a) = \zeta(\ol{a}^{-1})$; since the map $a\mapsto a^{q+1}$ maps $\mathbb{F}_{q^2}^\times$ onto $\mathbb{F}_q^\times$, this implies $\zeta$ is trivial on $\mathbb{F}_q^\times$.  

As before, we let $\mathbf{1}_I\in\textrm{c-ind}_I^G(\chi)$ denote the function with support in $I$, taking the value 1 at the identity.  We let $\mT_{n_s}$ (resp. $\mT_{n_{s'}}$) denote the endomorphism of $\textrm{c-ind}_{I}^G(\chi)$ sending $\mathbf{1}_I$ to the function with support $In_sI$ (resp. $In_{s'}I$), taking the value 1 at $n_s$ (resp. $n_{s'}$), on which $I$ acts by $\chi$.  In the notation of Section \ref{decompalg}, we have $$\mT_{n_s} = \T_{n_s}e_\chi,\qquad \mT_{n_{s'}} = \T_{n_{s'}}e_\chi.$$

\begin{prop}
\label{str}
 The algebra $\hh_C(G,\chi)$ is a noncommutative algebra, generated by $\mT_{n_s}$ and $\mT_{n_{s'}}$, subject to the relations
\begin{eqnarray*}
 (\mT_{n_s} + q^2)(\mT_{n_s} - q) & = & 0\\
 (\mT_{n_{s'}} + 1)(\mT_{n_{s'}} - q) & = & 0.
\end{eqnarray*}
The center $\mathcal{Z}_\chi$ is generated by $Z = \mT_{n_s}(\mT_{n_{s'}} - (q - 1)) + \mT_{n_{s'}}(\mT_{n_s} - (q - q^2))$.  We have an isomorphism of algebras $$\hh_C(G,\chi)\cong C\langle X,Y\rangle/(X^2 + (q^2 - q)X - q^3,Y^2 + (1-q)Y - q),$$ sending $\mT_{n_s}$ to $X$ and $\mT_{n_{s'}}$ to $Y$.  Here $C\langle X,Y \rangle$ denotes the noncommutative polynomial algebra in two variables over $C$.  
\end{prop}

\begin{proof}
 We shall prove Propositions \ref{triv} and \ref{str} simultaneously, assuming only that $\chi^s = \chi$ (recall that this implies that $\zeta$ is trivial on $\mathbb{F}_q^\times$).  By Frobenius Reciprocity, we may view elements of $\hh_C(G,\chi)$ as functions $\varphi:G\rightarrow C$ satisfying $$\varphi(igi') = \chi(i)\varphi(g)\chi(i')$$ for $g\in G, i,i'\in I$ (cf. \cite{BL94}, Proposition 5(1)).  If $\mT_{\varphi_1}, \mT_{\varphi_2}$ are the endomorphisms associated to $\varphi_1,\varphi_2$, respectively, then the composition product on $\hh_C(G,\chi)$ gives $\mT_{\varphi_1}\mT_{\varphi_2} = \mT_{\varphi_1*\varphi_2},$ where $$\varphi_1*\varphi_2(g) = \sum_{h\in G/I}\varphi_1(h)\varphi_2(h^{-1}g)$$
(see \emph{loc. cit.}, Proposition 5(3)).

Assume that $\varphi$ has support in $IwI$, where $w\in W_{\textrm{aff}}$.  Let $w = s_1s_2\cdots s_k$ be a reduced word expression for $w$, where $s_i\in \{s,s'\}$, and let $\varphi_{n_s}$ (resp. $\varphi_{n_{s'}}$) be the function with support in $In_sI$ (resp. $In_{s'}I$) taking the value 1 at $n_s$ (resp. $n_{s'}$).  We claim that $\varphi$ is a scalar multiple of $\varphi_{n_{s_1}}*\varphi_{n_{s_2}}*\ldots *\varphi_{n_{s_k}}$.  Indeed, the definition of the convolution product shows that $\textrm{supp}(\varphi_1*\varphi_2)\subset \textrm{supp}(\varphi_1)\textrm{supp}(\varphi_2)$, and therefore
\begin{eqnarray*}
\textrm{supp}(\varphi_{n_{s_1}}*\varphi_{n_{s_2}}*\ldots *\varphi_{n_{s_k}}) & \subset & \textrm{supp}(\varphi_{n_{s_1}})\textrm{supp}(\varphi_{n_{s_2}})\cdots\textrm{supp}(\varphi_{n_{s_k}})\\
 & = & In_{s_1}In_{s_2}I\cdots In_{s_k}I\\
 & = & IwI,
\end{eqnarray*}
where the last equality follows from \cite{AB08}, Prop. 6.36(4).  An elementary inductive argument shows that $\varphi_{n_{s_1}}*\varphi_{n_{s_2}}*\ldots *\varphi_{n_{s_k}}\neq 0$, which implies that $\hh_C(G,\chi)$ is generated as an algebra by $\mT_{n_s}$ and $\mT_{n_{s'}}$.

It suffices to determine the relations among $\mT_{n_s}$ and $\mT_{n_{s'}}$.  Note first that $\mT_{n_s}$ (resp. $\mT_{n_{s'}}$) preserves the space $\textnormal{c-ind}_{I}^{K}(\chi)\subset\textnormal{c-ind}_{I}^{G}(\chi)$ (resp. $\textnormal{c-ind}_{I}^{K'}(\chi)\subset\textnormal{c-ind}_{I}^{G}(\chi)$).  Additionally, we have the following identity, valid for $y\neq 0$:
\begin{equation}\label{inv}
u^-(x,y)  =  u\left(-\ol{x}\ \ol{y}^{-1},y^{-1}\right)h_s(-\ol{y}^{-1}\sqrt{\epsilon})n_su(-\ol{x}y^{-1},y^{-1}),
\end{equation}
where $h_s:E^\times\rightarrow T$ is the homomorphism of the remark following Theorem \ref{propstr}.  Using this notation, Proposition 3.18 of \cite{CL76} implies
\begin{align*}
 \mT_{n_s}^2 & = |\mathbb{U}| + \left(\sum_{u(x,y)\in \mathbb{U}\smallsetminus \{1\}}\chi^{-1}(h_s(\ol{y}^{-1}\sqrt{\epsilon}^{-1}))\right)\mT_{n_s}\\
 & = q^3 + \left(\sum_{\sub{x,y\in \mathbb{F}_{q^2},y\neq 0}{x\ol{x} + y + \ol{y} = 0}}\zeta(\ol{y}\sqrt{\epsilon})\right)\mT_{n_s}\\
 & = q^3 + \left(\sum_{y\in\mathbb{F}_{q}^\times}\zeta(y) + (q + 1)\left(\sum_{y\in\mathbb{F}_{q^2}^\times}\zeta(y) - \sum_{y\in \mathbb{F}_{q}^\times}\zeta(y)\right)\right)\mT_{n_s}\\
 & = \begin{cases}q^3 + (q^3 - 1)\mT_{n_s} & \textnormal{if}~\zeta~\textnormal{is trivial,}\\ q^3 + (q - q^2)\mT_{n_s} & \textnormal{if}~\zeta~\textnormal{is nontrivial.}\end{cases}\\
 \mT_{n_{s'}}^2 & = |\mathbb{U}'| + \left(\sum_{u^-(0,\varpi y)\in \mathbb{U}'\smallsetminus \{1\}}\chi^{-1}(h_s(y\sqrt{\epsilon}^{-1}))\right)\mT_{n_s}\\
 & = q + \left(\sum_{y\in\mathbb{F}_{q}^\times}\zeta^{-1}(y)\right)\mT_{n_{s'}}\\
 & = q + (q - 1)\mT_{n_{s'}}.
\end{align*}
This shows that any element of $\hh_C(G,\chi)$ may be written as a linear combination of monomials in $\mT_{n_s}$ and $\mT_{n_{s'}}$, with alternating terms.  Given two distinct such monomials, the two functions in $\textnormal{c-ind}_I^G(\chi)$ obtained by applying these operators to $\mathbf{1}_I$ will have disjoint support (this follows from the fact that $(I,N)$ is a BN pair).  Therefore, the quadratic relations are the only relations satisfied by $\mT_{n_s}$ and $\mT_{n_{s'}}$, and we obtain the isomorphism of $\hh_C(G,\chi)$ with a quotient of a noncommutative polynomial algebra.

Now, it is an elementary computation to check that $Z\in \mathcal{Z}_\chi$. To verify the claim about the centers $\mathcal{Z}_\chi$ of the algebras $\hh_C(G,\chi)$ in general, we first note that any central element, when viewed as a polynomial in $\mT_{n_s}$ and $\mT_{n_{s'}}$, must be of even degree.  Moreover, the quadratic relations imply that the coefficients of the two highest even-degree terms must be equal.   Let $\mathcal{Y}\in \mathcal{Z}_\chi$ be of degree $2k$.  Then there exists some $c\in C$ such that $\mathcal{Y} - cZ^k$ is of strictly smaller degree.  By induction on degree, $\mathcal{Y} - cZ^k$ is a polynomial in $Z$, and therefore $\mathcal{Z}_\chi = C[Z]$.  
\end{proof}

As before, we can now classify the finite-dimensional simple right $\hh_C(G,\chi)$-modules.  

\begin{defn}
\begin{enumerate}[(i)]
\item Let $(\theta,\theta')\in \{-q^2, q\}\times\{-1,q\}$.  We define the characters $\mu_{\theta,\theta'}:\hh_C(G,\chi)\rightarrow C$ by $$\mT_{n_s} \mapsto \theta,\quad \mT_{n_{s'}} \mapsto \theta'.$$
The central element $Z$ maps to $\theta(\theta' - q + 1) + \theta'(\theta - q + q^2)\in \{q^3 + q,-2q^2\}$.  
\item Let $\langle v_1, v_2\rangle_C$ be a two-dimensional vector space over $C$, and let $\lambda\in C$.  We define $M(\lambda)$ to be the following right $\hh_C(G,\chi)$-module:
\begin{center}
\begin{tabular}{cclccl}
 $v_1\cdot \mT_{n_s}$ & = & $-q^2v_1,$ & $v_1\cdot \mT_{n_{s'}}$ & = & $v_2$\\
 $v_2\cdot \mT_{n_s}$ & = & $(\lambda + q^2 - q^3)v_1 + q v_2,$ & $v_2\cdot \mT_{n_{s'}}$ & = & $qv_1 + (q-1)v_2$
\end{tabular}
\end{center}
The central element $Z$ acts by $\lambda$.  
\end{enumerate}
\end{defn}

Again, the action of $\hh_C(G,\chi)$ on $M(\lambda)$ is well-defined.  As the proofs of the following results are similar to their counterparts in the trivial case (cf. Propositions \ref{iwahorired} and \ref{exc}, and Theorem \ref{thmiw}), we omit them.  

\begin{prop}
 Assume $\ch(C) \neq p$. Then $M(\lambda)$ is reducible if and only if $\lambda = q^3 + q$ or $\lambda = -2q^2$.  In these cases, we have the following nonsplit exact sequences:
$$0\rightarrow \mu_{q,  q}\rightarrow M(q^3 + q)\rightarrow \mu_{-q^2, -1}\rightarrow 0$$
$$0\rightarrow \mu_{q, -1}\rightarrow M(-2q^2)\rightarrow \mu_{-q^2,q}\rightarrow 0$$
\end{prop}

%
%
%

\begin{prop}
 Assume $\ch(C) = p$.  Then $M(\lambda)$ is reducible if and only if $\lambda = 0$.  In this case the module decomposes as $M(0) \cong \mu_{0,0}\oplus\mu_{0,-1}$.  
\end{prop}


\begin{thm}\label{thmiw2}
 Every finite-dimensional simple right $\hh_C(G,\chi)$-module is isomorphic to either a character $\mu_{\theta,\theta'}, (\theta,\theta')\in\{-q^2, q\}\times\{-1,q\}$, or a module of the form $M(\lambda), \lambda\neq q^3 + q, -2q^2$.
\end{thm}


\subsection{The Regular Case}

We assume now that $\chi^s \neq \chi = \zeta\otimes\eta$.  In this case we have nontrivial intertwining maps between $\textrm{c-ind}_I^G(\chi)$ and $\textrm{c-ind}_I^G(\chi^s)$, and we are led to consider the algebra 
\begin{eqnarray*}
\hh_C(G,\gamma_\chi) & = & \hh_C(G,\chi\oplus\chi^s)\\
 & = & \hh_C(G,\chi)\oplus \textrm{Hom}_G(\textrm{c-ind}_I^G(\chi), \textrm{c-ind}_I^G(\chi^s))\\
 & & \qquad\oplus\ \textrm{Hom}_G(\textrm{c-ind}_I^G(\chi^s), \textrm{c-ind}_I^G(\chi))\oplus\hh_C(G,\chi^s).
\end{eqnarray*}
We first determine the algebra $\hh_C(G,\chi)$.  For $n\in N$, we denote by $\mathbf{1}_{InI}\in \textrm{c-ind}_I^G(\chi)$ the function with support $InI$, taking the value 1 at $n$, on which $I$ acts by $\chi$ or $\chi^s$ (depending on the class of $n$ in $W$).  We let $\mT_{\alpha^{-1}}$ (resp. $\mT_\alpha$) denote the endomorphism of $\textrm{c-ind}_I^G(\chi)$ sending $\mathbf{1}_I$ to $\mathbf{1}_{I\alpha^{-1}I}$ (resp. $\mathbf{1}_{I\alpha I}$).

\begin{prop}
\label{thmreg}
 The algebra $\hh_C(G,\chi)$ is commutative, generated by $\mT_{\alpha^{-1}}$ and $\mT_{\alpha}$, with the relations 
$$\mT_{\alpha^{-1}}\mT_\alpha = \mT_\alpha\mT_{\alpha^{-1}} = q^4.$$
We have an isomorphism of algebras $\hh_C(G,\chi)\cong C[X,Y]/(XY - q^4)$, sending $\mT_{\alpha^{-1}}$ to $X$ and $\mT_\alpha$ to $Y$.
\end{prop}

\begin{proof}
 We adopt the same method as in the proof of \ref{str}, viewing elements of $\hh_C(G,\chi)$ as functions $\varphi$ on the double cosets $I\backslash G/I$.  In this case, however, the relation $\varphi(igi') = \chi(i)\varphi(g)\chi(i')$ shows that the functions $\varphi$ associated to elements of $\hh_C(G,\chi)$ are supported only on cosets of the form $I\alpha^nI$, $n\in \mathbb{Z}$.  If $\varphi$ has support in $I\alpha^{-n}I$ (resp. $I\alpha^nI$) with $n>0$, then $\varphi$ is a scalar multiple of $\varphi_{\alpha^{-1}}*\varphi_{\alpha^{-1}}*\ldots *\varphi_{\alpha^{-1}}$ (resp. $\varphi_\alpha *\varphi_\alpha *\ldots *\varphi_\alpha$), the convolution taken $n$ times.
  
It therefore suffices to compute the products $\varphi_{\alpha^{-1}}*\varphi_\alpha$ and $\varphi_\alpha *\varphi_{\alpha^{-1}}$.  We compute the first of these; the method of computation for the second is the same.  We have $\textrm{supp}(\varphi_{\alpha^{-1}}*\varphi_\alpha)\subset I\alpha^{-1}I\alpha I \subset I\sqcup In_sI\sqcup I\alpha^{-1}n_sI$, and since the convolution must have support on cosets of the form $I\alpha^nI$, we actually have $\textrm{supp}(\varphi_{\alpha^{-1}}*\varphi_\alpha)\subset I$.  Hence, we need only evaluate this function at 1.  We have:
\begin{equation*}
 \varphi_{\alpha^{-1}}*\varphi_\alpha(1)  =  \sum_{h\in I\alpha^{-1}I/I}\varphi_{\alpha^{-1}}(h)\varphi_\alpha(h^{-1}) = 
 \sum_{h\in I\alpha^{1}I/I} 1 =  q^{\ell(\alpha^{-1})} = q^4.
\end{equation*}\end{proof}

We now turn our attention to $\textrm{Hom}_G(\textrm{c-ind}_I^G(\chi), \textrm{c-ind}_I^G(\chi^s))$.  This has the structure of an $(\hh_C(G,\chi^s),\hh_C(G,\chi))$-bimodule, with the action given by post-composition and pre-composition, respectively.  By Frobenius Reciprocity we have 
\begin{center}
$\textrm{Hom}_G(\textrm{c-ind}_I^G(\chi), \textrm{c-ind}_I^G(\chi^s)) \cong \textrm{Hom}_I(\chi,\textrm{c-ind}_I^G(\chi^s)|_I)\cong \textrm{c-ind}_I^G(\chi^s)^{I,\chi}$,
\end{center}
 which has a basis consisting of the functions $\mathbf{1}_{In_s\alpha^nI}$ with support $In_s\alpha^nI$ and value $1$ at $n_s\alpha^n$, on which $I$ acts by $\chi$.  We let $\mS_{n,\chi}$ denote the homomorphism sending $\mathbf{1}_I\in \textrm{c-ind}_I^G(\chi)$ to $\mathbf{1}_{In_s\alpha^nI}\in \textrm{c-ind}_I^G(\chi^s)$, and append a $\chi$ (or $\chi^s$) to the parameters for the operator $\mT_{\alpha}$ (or $\mT_{\alpha^{-1}}$) to specify the Hecke algebra to which it corresponds.  In the notation of Section \ref{decompalg}, we have 
$$\mS_{0,\chi} = \T_{n_s}e_\chi,\qquad \mS_{0,\chi^s} = \T_{n_s}e_{\chi^s},$$
$$\mS_{-1,\chi} = \T_{n_{s'}}e_\chi,\qquad \mS_{-1,\chi^s} = \T_{n_{s'}}e_{\chi^s}.$$

We note that the set 
\begin{equation}\label{basis}
\{\textnormal{id}_{\chi},~ \mT_{\alpha,\chi}^m,~ \mT_{\alpha^{-1},\chi}^m,~ \textnormal{id}_{\chi^s},~ \mT_{\alpha,\chi^s}^m,~ \mT_{\alpha^{-1},\chi^s}^m,~ \mS_{n,\chi},~ \mS_{n,\chi^s}\}_{m > 0, n\in \mathbb{Z}}
\end{equation}
 forms a basis for $\hh_C(G,\gamma_\chi)$, where $\textnormal{id}_{\chi}$ (resp. $\textnormal{id}_{\chi^s}$) denotes the identity element of $\hh_C(G,\chi)$ (resp. $\hh_C(G,\chi^s)$).

\begin{prop}\label{mod}
 We have the following relations for the $(\hh_C(G,\chi^s),\hh_C(G,\chi))$-bimodule $\textnormal{Hom}_G(\textnormal{c-ind}_I^G(\chi), \textnormal{c-ind}_I^G(\chi^s))$:
\begin{eqnarray*}
 \mT_{\alpha^{-1},\chi^s}\mS_{n,\chi} = \mS_{n,\chi}\mT_{\alpha,\chi} & = & \begin{cases}\mS_{n+1,\chi} & n\geq 0\\ q\mS_{n+1,\chi} & n = -1\\ q^4\mS_{n+1,\chi} & n\leq -2\end{cases}\\
 \mT_{\alpha,\chi^s}\mS_{n,\chi} = \mS_{n,\chi}\mT_{\alpha^{-1},\chi} & = & \begin{cases}q^4\mS_{n-1,\chi} & n\geq 1\\ q^3\mS_{n-1,\chi} & n = 0\\ \mS_{n-1,\chi} & n\leq -1   \end{cases}
\end{eqnarray*}
In particular, $\textnormal{Hom}_G(\textnormal{c-ind}_I^G(\chi), \textnormal{c-ind}_I^G(\chi^s))$ is generated as a module by $\mS_{0,\chi}$ and $\mS_{-1,\chi}$.  
\end{prop}

\begin{proof}
We will need the following coset decompositions:
\begin{center}
\begin{align}
I\alpha I & =  \bigsqcup_{\sub{x\in \mathbb{F}_{q^2},y\in \mathfrak{o}_E/\mathfrak{p}_E^2}{x\ol{x} + y + \ol{y} = 0}}I\alpha u(x,y)\label{firsteq}\\
I\alpha^{-1}I & =  \bigsqcup_{\sub{x\in \mathbb{F}_{q^2},y\in \mathfrak{o}_E/\mathfrak{p}_E^2}{\varpi x\ol{x} + y + \ol{y} = 0}}I\alpha^{-1}u^-(\varpi x,\varpi y)\\
In_s\alpha^nI & =  \bigsqcup_{\sub{x\in \mathfrak{o}_E/\mathfrak{p}_E^{n+1},y\in \mathfrak{o}_E/\mathfrak{p}_E^{2n+1}}{x\ol{x} + y + \ol{y} = 0}}In_s\alpha^nu(x,y)  & \textrm{if}\ n\geq 0 \label{decs}\\
In_s\alpha^nI & =  \bigsqcup_{\sub{x\in \mathfrak{o}_E/\mathfrak{p}_E^{-n-1},y\in \mathfrak{o}_E/\mathfrak{p}_E^{-2n-1}}{\varpi x\ol{x} + y + \ol{y} = 0}}In_s\alpha^nu^-(\varpi x,\varpi y) & \textrm{if}\ n < 0\label{decs'}
\end{align}
\end{center}
In order to compute $\mS_{n,\chi}\mT_{\alpha,\chi}$, it suffices to know its action on the function $\mathbf{1}_I\in \textrm{c-ind}_I^G(\chi)$.  The definitions of $\mS_{n,\chi}$ and $\mT_{\alpha,\chi}$ show that the image will have support contained in $In_s\alpha^nI\alpha I$.  Using Proposition 6.36 and Exercise 6.37 of \cite{AB08}, we see that this product of double cosets is equal to $In_s\alpha^{n+1}I$ (if $\ell(n_s\alpha^{n+1}) = \ell(n_s\alpha^n) + \ell(\alpha)$), or is contained in $In_s\alpha^{n+1}I \sqcup I\alpha^{-n}I \sqcup I\alpha^{-n-1}I$ (if $\ell(n_s\alpha^{n+1}) \neq \ell(n_s\alpha^n) + \ell(\alpha)$).  Since the support of $\mS_{n,\chi}\mT_{\alpha,\chi}(\mathbf{1}_I)$ must be of the form $In_s\alpha^mI$, we see that in both cases the support is contained in $In_s\alpha^{n+1}I$, and therefore it suffices to evaluate the function at $n_s\alpha^{n+1}$.  This gives (using equation \eqref{firsteq})
\begin{eqnarray*}
\mS_{n,\chi}\mT_{\alpha,\chi}(\mathbf{1}_I)(n_s\alpha^{n+1}) & = & \mS_{n,\chi}\left(\sum_{\sub{x\in \mathbb{F}_{q^2},y\in \mathfrak{o}_E/\mathfrak{p}_E^2}{x\ol{x} + y + \ol{y} = 0}} u(-x,\ol{y})\alpha^{-1}.\mathbf{1}_I\right)(n_s\alpha^{n+1})\\
 & = & \sum_{\sub{x\in \mathbb{F}_{q^2},y\in \mathfrak{o}_E/\mathfrak{p}_E^2}{x\ol{x} + y + \ol{y} = 0}}u(-x,\ol{y})\alpha^{-1}.\mathbf{1}_{In_s\alpha^nI}(n_s\alpha^{n+1})\\
 & = & \sum_{\sub{x\in \mathbb{F}_{q^2},y\in \mathfrak{o}_E/\mathfrak{p}_E^2}{x\ol{x} + y + \ol{y} = 0}}\mathbf{1}_{In_s\alpha^nI}(n_s\alpha^{n+1}u(-x,\ol{y})\alpha^{-1})\\
 & \stackrel{\textnormal{eq.}~ \eqref{inv}}{=} & \begin{cases}1 & n\geq 0\\q & n = -1\\ q^4 & n\leq -2.\end{cases}
\end{eqnarray*}
The other relations are proved similarly, using the various coset decompositions above.  
\end{proof}

\begin{cor}
 As a right $\hh_C(G,\chi)$-module, we have 
$$\textnormal{Hom}_G(\textnormal{c-ind}_I^G(\chi), \textnormal{c-ind}_I^G(\chi^s)) \cong (\hh_C(G,\chi)\oplus\hh_C(G,\chi))/((\mT_{\alpha^{-1},\chi},-q^3),(-q,\mT_{\alpha,\chi})),$$
the isomorphism sending $\mS_{0,\chi}$ to $(1,0)$ and $\mS_{-1,\chi}$ to $(0,1)$.  

Likewise, as a left $\hh_C(G,\chi^s)$-module, we have
$$\textnormal{Hom}_G(\textnormal{c-ind}_I^G(\chi), \textnormal{c-ind}_I^G(\chi^s)) \cong (\hh_C(G,\chi^s)\oplus\hh_C(G,\chi^s))/((\mT_{\alpha,\chi^s},-q^3),(-q,\mT_{\alpha^{-1},\chi^s})),$$
the isomorphism sending $\mS_{0,\chi}$ to $(1,0)$ and $\mS_{-1,\chi}$ to $(0,1)$.  
\end{cor}

In addition to the bimodule structure on $\textrm{Hom}_G(\textrm{c-ind}_I^G(\chi), \textrm{c-ind}_I^G(\chi^s))$, we also have a composition product between elements $\mS_{n,\chi^s}\in\textrm{Hom}_G(\textrm{c-ind}_I^G(\chi^s), \textrm{c-ind}_I^G(\chi))$ and elements $\mS_{m,\chi}\in\textrm{Hom}_G(\textrm{c-ind}_I^G(\chi), \textrm{c-ind}_I^G(\chi^s))$.  The product of two such homomorphisms will be an element of $\hh_C(G,\chi)$.  

\begin{prop}
\label{mod2}
 The composition $\mS_{n,\chi^s}\mS_{m,\chi}$ has the following property:
$$\mS_{n,\chi^s}\mS_{m,\chi} = \begin{cases}\zeta(-1)q^{3 + 4\min(n,m)}\mT_{\alpha,\chi}^{\max(0,m-n)}\mT_{\alpha^{-1},\chi}^{\max(n-m,0)} & n,m\geq 0\\ \zeta(-1)q^{1+4\min(-n-1,-m-1)}\mT_{\alpha,\chi}^{\max(0,m-n)}\mT_{\alpha^{-1},\chi}^{\max(n-m,0)} & n,m<0\\ \zeta(-1)\mT_{\alpha,\chi}^{m-n} & n<0, m\geq 0\\ \zeta(-1)\mT_{\alpha^{-1},\chi}^{n-m} & m<0, n\geq 0.\end{cases}$$
\end{prop}

\begin{proof}
 By Proposition \ref{mod}, it suffices to compute the four products $\mS_{0,\chi^s}\mS_{0,\chi},\ \mS_{-1,\chi^s}\mS_{-1,\chi},$ $\ \mS_{-1,\chi^s}\mS_{0,\chi}$ and $\mS_{0,\chi^s}\mS_{-1,\chi}$.  The first two products follow again from Proposition 3.18 of \cite{CL76}.  We prove the third relation; the last follows similarly.  The method of proof is the same as in the proof of Proposition \ref{mod}, this time using equations \eqref{decs} and \eqref{decs'} for $n = 0$ and $n = -1$.  The definition of $\mS_{-1,\chi^s}$ and $\mS_{0,\chi}$ and properties of BN pairs show that the function $\mS_{-1,\chi^s}\mS_{0,\chi}(\mathbf{1}_I)$ will have support in $I \alpha I$.  This gives
\begin{eqnarray*}
 \mS_{-1,\chi^s}\mS_{0,\chi}(\mathbf{1}_I)(\alpha) & = & \mS_{-1,\chi^s}\left(\sum_{\sub{x,y\in \mathbb{F}_{q^2}}{x\ol{x} + y + \ol{y} = 0}}u(-x,\ol{y})n_s^{-1}.\mathbf{1}_I\right)(\alpha)\\
 & = & \sum_{\sub{x,y\in \mathbb{F}_{q^2}}{x\ol{x} + y + \ol{y} = 0}} u(-x,\ol{y})n_s^{-1}.\mathbf{1}_{In_{s'}I}(\alpha)\\
 & = & \sum_{\sub{x,y\in \mathbb{F}_{q^2}}{x\ol{x} + y + \ol{y} = 0}} \mathbf{1}_{In_{s'}I}(\alpha u(-x,\ol{y})n_s^{-1})\\
 & \stackrel{\textnormal{eq.}~ \eqref{inv}}{=} & \zeta(-1).
\end{eqnarray*}
\end{proof}

Combining Propositions \ref{thmreg}, \ref{mod}, and \ref{mod2}, we now have a full description of the algebra structure of $\hh_C(G,\gamma_\chi)$.  When $\ch(C) = p$, there is a more elegant presentation of $\hh_C(G,\gamma_\chi)$:

\begin{cor}
\label{twistmatalg}
 Assume $\ch(C) = p$, let $R = C[X,Y]/(XY)$, $\mathfrak{m} = (X,Y)\subset R$, and let $M$ denote the $R$-module $R/(Y)\oplus R/(X)\cong C[X]\oplus C[Y]$.  The ideal $\mathfrak{m}$ acts on $M$, and we lift the image of this action to $R$ by sending $(f(X),g(Y))\in \mathfrak{m}M$ to $f(X) + g(Y)\in R$.  We then have
 $$\hh_C(G,\gamma_\chi)\cong \begin{pmatrix}R & M\\ \mathfrak{m} & R \end{pmatrix}.$$
The isomorphism is given by
\begin{center}
\begin{tabular}{cclccll}
 $\mT_{\alpha,\chi}$ & $\mapsto$ & $\begin{pmatrix} Y & 0 \\ 0 & 0\end{pmatrix}$ & $\mT_{\alpha^{-1},\chi}$ & $\mapsto$ & $\begin{pmatrix} X & 0 \\ 0 & 0\end{pmatrix}$ & \\
 $\mT_{\alpha,\chi^s}$ & $\mapsto$ & $\begin{pmatrix} 0 & 0 \\ 0 & X\end{pmatrix}$ & $\mT_{\alpha^{-1},\chi^s}$ & $\mapsto$ & $\begin{pmatrix} 0 & 0 \\ 0 & Y\end{pmatrix}$ & \\
 $\mS_{n,\chi}$ & $\mapsto$ & $\begin{pmatrix} 0 & 0 \\ \zeta(-1)Y^{n + 1} & 0\end{pmatrix}$ & $\mS_{n,\chi^s}$ & $\mapsto$ & $\begin{pmatrix} 0 & (X^n,0) \\ 0 & 0\end{pmatrix}$ & for $n\geq 0$\\
 $\mS_{n,\chi}$ & $\mapsto$ & $\begin{pmatrix} 0 & 0 \\ \zeta(-1)X^{-n} & 0\end{pmatrix}$ & $\mS_{n,\chi^s}$ & $\mapsto$ & $\begin{pmatrix} 0 & (0,Y^{-n-1}) \\ 0 & 0\end{pmatrix}$ & for $n<0$.
\end{tabular}
\end{center}
The center $\mathcal{Z}_{\gamma_\chi}$ of $\hh_C(G,\gamma_\chi)$ consists of all elements of the form $$h(\mT_{\alpha,\chi},\mT_{\alpha^{-1},\chi}) + h(\mT_{\alpha^{-1},\chi^s}, \mT_{\alpha,\chi^s}),$$ where $h$ is a polynomial of two variables.
\end{cor}

\begin{remark}
 The description of the center $\mathcal{Z}_{\gamma_\chi}$ of $\hh_C(G,\gamma_\chi)$ is the same for the case $\ch(C)\neq p$.
\end{remark}

\begin{proof}
 It only remains to verify the claim about the center.  This follows easily from the isomorphism above.  \end{proof}

We may now classify finite-dimensional simple modules for the algebra $\hh_C(G,\gamma_\chi)$.  

\begin{prop}
 Assume $\ch(C)\neq p$.  Then $\hh_C(G,\gamma_\chi)$ admits no characters.
\end{prop}

\begin{proof}
 Let $\mu$ be a character of $\hh_C(G,\gamma_\chi)$.  Since $\mS_{n,\chi}^2 = \mS_{n,\chi^s}^2 = 0$, we must have $\mu(\mS_{n,\chi}) = \mu(\mS_{n,\chi^s}) = 0$ for every $n\in \mathbb{Z}$.  Proposition \ref{mod2} now implies that all elements of $\hh_C(G,\chi)$ and $\hh_C(G,\chi^s)$ map to 0.  This gives a contradiction, since $1 = \mu(\textrm{id}_{\textrm{c-ind}_I^G(\gamma_\chi)}) = \mu(\textnormal{id}_{\chi} + \textnormal{id}_{\chi^s}) = 0$.  
\end{proof}

\begin{defn}
Assume $\ch(C) = p$, and let $i\in \{0,1\}$.  We define $\mu_i: \hh_C(G,\gamma_\chi)\rightarrow C$ to be the character for which $$\textnormal{id}_{\chi^{s^i}} \mapsto 1$$ and every other basis element in the set \eqref{basis} maps to 0.  
\end{defn}

\begin{prop}
 Assume $\ch(C)= p$.  Then the characters of $\hh_C(G,\gamma_\chi)$ are exactly $\mu_0$ and $\mu_1$.  
\end{prop}

\begin{proof}
 As in the characteristic prime-to-$p$ case, we use Propositions \ref{mod} and \ref{mod2} to conclude that every basis element besides $\textnormal{id}_{\chi}$ and $\textnormal{id}_{\chi^s}$ must map to zero.  Since $\textnormal{id}_{\chi} + \textnormal{id}_{\chi^s} = \textrm{id}_{\textrm{c-ind}_I^G(\gamma_\chi)}$ and $\textnormal{id}_{\chi} \textnormal{id}_{\chi^s} = 0$, we see that the characters must be exactly those stated.
\end{proof}

We now turn our attention to modules of dimension greater than one.  We first assume that $\ch(C) \neq p$.  Fix a choice of square root $\sqrt{\zeta(-1)}$ of $\zeta(-1)$, and let 
$$\mathcal{A} = \sqrt{\zeta(-1)}(\mS_{0,\chi}+\mS_{-1,\chi^s});$$  
then $\mathcal{A}^2 = \mT_{\alpha,\chi} + \mT_{\alpha^{-1},\chi^s}$, and $\hh_C(G,\gamma_\chi)$ is free of rank two over $\mathcal{Z}_{\gamma_\chi}[\mathcal{A}]$, with basis $\{\textnormal{id}_{\chi}, \textnormal{id}_{\chi^s}\}$.  

Let $\lambda\in C^\times$, and let $\mu_{\lambda,\sqrt{\lambda}}$ denote the character of $\mathcal{Z}_{\gamma_\chi}[\mathcal{A}]$ defined by 
$$(\mT_{\alpha,\chi} + \mT_{\alpha^{-1},\chi^s}) \mapsto \lambda,\quad (\mT_{\alpha,\chi^s} + \mT_{\alpha^{-1},\chi}) \mapsto q^4\lambda^{-1},\qquad \mathcal{A} \mapsto \sqrt{\lambda}.$$
We consider the induced representation $\mu_{\lambda,\sqrt{\lambda}}\otimes_{\mathcal{Z}_{\gamma_\chi}[\mathcal{A}]}\hh_C(G,\gamma_\chi)$.  Since the algebra $\hh_C(G,\gamma_\chi)$ admits no characters, this immediately implies that this module is simple.  

\begin{lemma}\label{Mlambda}
 The (isomorphism class of the) representation $\mu_{\lambda,\sqrt{\lambda}}\otimes_{\mathcal{Z}_{\gamma_\chi}[\mathcal{A}]}\hh_C(G,\gamma_\chi)$ is independent of the choice of square roots $\sqrt{\lambda}$ and $\sqrt{\zeta(-1)}$.
\end{lemma}

\begin{proof}
 Let $\langle v\rangle_C$ denote the underlying space of $\mu_{\lambda,\sqrt{\lambda}}$, so that $\mu_{\lambda,\sqrt{\lambda}}\otimes_{\mathcal{Z}_{\gamma_\chi}[\mathcal{A}]}\hh_C(G,\gamma_\chi)$ is spanned by $\{v\otimes \textnormal{id}_{\chi}, v\otimes \textnormal{id}_{\chi^s}\}$.  The action of $\mathcal{Z}_{\gamma_\chi}[\mathcal{A}]$ on the vector $v\otimes(\textnormal{id}_{\chi} - \textnormal{id}_{\chi^s})$ shows that $\mu_{\lambda,-\sqrt{\lambda}}$ is contained in $\mu_{\lambda,\sqrt{\lambda}}\otimes_{\mathcal{Z}_{\gamma_\chi}[\mathcal{A}]}\hh_C(G,\gamma_\chi)|_{\mathcal{Z}_{\gamma_\chi}[\mathcal{A}]}$.  By Frobenius Reciprocity we have 
\begin{eqnarray*}
 \{0\} & \neq & \textrm{Hom}_{\mathcal{Z}_{\gamma_\chi}[\mathcal{A}]}(\mu_{\lambda,-\sqrt{\lambda}},~\mu_{\lambda,\sqrt{\lambda}}\otimes_{\mathcal{Z}_{\gamma_\chi}[\mathcal{A}]}\hh_C(G,\gamma_\chi)|_{\mathcal{Z}_{\gamma_\chi}[\mathcal{A}]})\\
 & \cong & \textrm{Hom}_{\hh_C(G,\gamma_\chi)}(\mu_{\lambda,-\sqrt{\lambda}}\otimes_{\mathcal{Z}_{\gamma_\chi}[\mathcal{A}]}\hh_C(G,\gamma_\chi),~\mu_{\lambda,\sqrt{\lambda}}\otimes_{\mathcal{Z}_{\gamma_\chi}[\mathcal{A}]}\hh_C(G,\gamma_\chi)).
\end{eqnarray*}
  As both modules are simple, the result follows.
\end{proof}

Given this lemma, we define $M(\lambda) := \mu_{\lambda,\sqrt{\lambda}}\otimes_{\mathcal{Z}_{\gamma_\chi}[\mathcal{A}]}\hh_C(G,\gamma_\chi)$.  Examining central characters, we see that the modules $M(\lambda)$ are pairwise nonisomorphic for distinct values of $\lambda$.  

\begin{thm}\label{charCneq0reg}
 Assume $\ch(C) \neq p$.  Every finite-dimensional simple right $\hh_C(G,\gamma_\chi)$-module is isomorphic to one of the form $M(\lambda), \lambda\in C^\times$.  
\end{thm}

\begin{proof}
 Assume $M$ is a nonzero simple right module, and assume that $M|_{\mathcal{Z}_{\gamma_\chi}[\mathcal{A}]}$ contains a character $\mu_{\lambda,\sqrt{\lambda}}$.  Frobenius Reciprocity gives $$\{0\}\neq \textrm{Hom}_{\mathcal{Z}_{\gamma_\chi}[\mathcal{A}]}(\mu_{\lambda,\sqrt{\lambda}},M|_{\mathcal{Z}_{\gamma_\chi}[\mathcal{A}]})\cong \textrm{Hom}_{\hh_C(G,\gamma_\chi)}(M(\lambda), M),$$ which implies $M(\lambda)\cong M$ by simplicity of $M(\lambda)$ and $M$.  
\end{proof}

Assume now that $\ch(C) = p$, fix a choice of square root $\sqrt{\zeta(-1)}$ of $\zeta(-1)$, and let 
$$\mathcal{A}_1 = \sqrt{\zeta(-1)}(\mS_{0,\chi} + \mS_{-1,\chi^s}),\qquad \mathcal{A}_2 = \sqrt{\zeta(-1)}(\mS_{0,\chi^s} + \mS_{-1,\chi}).$$
Note that $\mathcal{A}_1\mathcal{A}_2 = \mathcal{A}_2\mathcal{A}_1 = 0$, $\mathcal{A}_1^2 = \mT_{\alpha,\chi} + \mT_{\alpha^{-1},\chi^s}$, and $\mathcal{A}_2^2 = \mT_{\alpha^{-1},\chi} + \mT_{\alpha,\chi^s}$.  The algebra $\hh_C(G,\gamma_\chi)$ is free of rank two over $\mathcal{Z}_{\gamma_\chi}[\mathcal{A}_1,\mathcal{A}_2]$, with basis $\{\textnormal{id}_{\chi},\textnormal{id}_{\chi^s}\}$.  

Let $\lambda,\lambda'\in C$ be such that $\lambda\lambda' = 0$, and let $\mu_{\lambda,\lambda',\sqrt{\lambda},\sqrt{\lambda'}}$ denote the character of $\mathcal{Z}_{\gamma_\chi}[\mathcal{A}_1,\mathcal{A}_2]$ defined by

\begin{center}
\begin{tabular}{rclrcl}
 $(\mT_{\alpha,\chi} + \mT_{\alpha^{-1},\chi^s})$ & $\mapsto$ & $\lambda$, & $(\mT_{\alpha^{-1},\chi} + \mT_{\alpha,\chi^s})$ & $\mapsto$ & $\lambda',$\\
 $\mathcal{A}_1$ & $\mapsto$ & $\sqrt{\lambda}$, & $\mathcal{A}_2$ & $\mapsto$ & $\sqrt{\lambda'}.$
\end{tabular}
\end{center}

We consider the induced representation $\mu_{\lambda,\lambda',\sqrt{\lambda},\sqrt{\lambda'}}\otimes_{\mathcal{Z}_{\gamma_\chi}[\mathcal{A}_1,\mathcal{A}_2]}\hh_C(G,\gamma_\chi)$.  

\begin{prop}
 The module $\mu_{\lambda,\lambda',\sqrt{\lambda},\sqrt{\lambda'}}\otimes_{\mathcal{Z}_{\gamma_\chi}[\mathcal{A}_1,\mathcal{A}_2]}\hh_C(G,\gamma_\chi)$ is reducible if and only if $(\lambda,\lambda') = (0,0)$.  In this case, we have $$\mu_{0,0,0,0}\otimes_{\mathcal{Z}_{\gamma_\chi}[\mathcal{A}_1,\mathcal{A}_2]}\hh_C(G,\gamma_\chi)\cong \mu_0\oplus \mu_1.$$
\end{prop}

\begin{proof}
 Assume that $\mu_{\lambda,\lambda',\sqrt{\lambda},\sqrt{\lambda'}}\otimes_{\mathcal{Z}_{\gamma_\chi}[\mathcal{A}_1,\mathcal{A}_2]}\hh_C(G,\gamma_\chi)$ is reducible, so that it contains either $\mu_0$ or $\mu_1$.  In either case, both $\mathcal{A}_1$ and $\mathcal{A}_2$ must act by $0$, and therefore $(\lambda,\lambda') = (0,0)$.  The action of $\textnormal{id}_{\chi}$ and $\textnormal{id}_{\chi^s}$ show that if $\langle v\rangle_C$ denotes the underlying space of $\mu_{0,0,0,0}$, then $\langle v\otimes \textnormal{id}_{\chi}\rangle_C \cong \mu_0$ and $\langle v\otimes \textnormal{id}_{\chi^s}\rangle_C \cong \mu_1$ as $\hh_C(G,\gamma_\chi)$-modules.
\end{proof}

\begin{lemma}
 The (isomorphism class of the) representation $\mu_{\lambda,\lambda',\sqrt{\lambda},\sqrt{\lambda'}}\otimes_{\mathcal{Z}_{\gamma_\chi}[\mathcal{A}_1,\mathcal{A}_2]}\hh_C(G,\gamma_\chi)$ is independent of the choice of square roots $\sqrt{\lambda}, \sqrt{\lambda'}$, and $\sqrt{\zeta(-1)}$.
\end{lemma}

\begin{proof}
The proof is similar to the proof of Lemma \ref{Mlambda}.  
\end{proof}

Given this lemma, we can define $M(\lambda,\lambda') := \mu_{\lambda,\lambda',\sqrt{\lambda},\sqrt{\lambda'}}\otimes_{\mathcal{Z}_{\gamma_\chi}[\mathcal{A}_1,\mathcal{A}_2]}\hh_C(G,\gamma_\chi)$.  By examining central characters, we see that the modules $M(\lambda,\lambda')$ are pairwise nonisomorphic for distinct pairs $(\lambda,\lambda')$.

\begin{thm}\label{thmneqiw}
 Assume $\ch(C) = p$.  Every finite-dimensional simple right $\hh_C(G,\gamma_\chi)$-module is isomorphic to either a character $\mu_0$ or $\mu_1$, or a module of the form $M(\lambda,\lambda')$ with $\lambda\lambda' = 0, (\lambda,\lambda')\neq (0,0)$.  
\end{thm}

\begin{proof}
Again, the proof is similar to the prime-to-$p$ case (cf. Theorem \ref{charCneq0reg}).  
\end{proof}

We conclude with one final definition.

\begin{defn}\label{append}
 Let $\chi:H\rightarrow C^\times$ be an arbitrary character, and let $M$ be a finite-dimensional simple module for $\hh_C(G,\gamma_\chi)$.  We append $\chi$ to the list of parameters of $M$, and use this notation to denote the corresponding module for $\hh_C(G,I(1))$, via the decomposition of Proposition \ref{orbits}.
\end{defn}

\begin{remark}
 The isomorphism in Corollary \ref{twistmatalg} depends on the \emph{ordered} pair $(\chi,\chi^s)$.  There is an obvious isomorphism of algebras $\hh_C(G,\chi\oplus\chi^s)\cong \hh_C(G,\chi^s\oplus\chi)$, which identifies simple modules.  In particular, the isomorphism gives $M(\lambda,\chi)\cong M(q^4\lambda^{-1},\chi^s)$ if $\ch(C)\neq p$, and  $\mu_{0,\chi}\cong\mu_{1,\chi^s},\ \mu_{1,\chi}\cong\mu_{0,\chi^s},\ M(\lambda,\lambda',\chi)\cong M(\lambda',\lambda,\chi^s)$ if $\ch(C) = p$.  
\end{remark}

\section{Principal Series and Supersingular Modules}\label{ssingmodssec}

\subsection{Principal Series} 
We shall assume from this point onwards that $C = \ol{\mathbb{F}}_p$, and that all representations are smooth $\ol{\mathbb{F}}_p$-representations.  We call such representations \emph{mod-$p$} or \emph{modular} representations.  In an attempt to understand supersingular representations of $G$ (cf. Introduction), we will make use of the functor sending a smooth representation $\pi$ to $\pi^{I(1)}$, called the functor of $I(1)$-invariants.  By Lemma 3(1) of \cite{BL94}, if $\pi$ is a nonzero smooth representation of $G$, then the $\hh_{\ol{\mathbb{F}}_p}(G,I(1))$-module $\pi^{I(1)}$ will also be nonzero.

Let $\widetilde{\chi} = \widetilde{\zeta}\otimes\widetilde{\eta}$ be a smooth character of the full torus $T$ of $G$, and consider the principal series representation $\textrm{ind}_B^G(\widetilde{\chi})$, where $B$ is the standard upper Borel subgroup of $G$, and $\widetilde{\zeta}$ and $\widetilde{\eta}$ are characters of $E^\times$ and $\textnormal{U}(1)(E/F)$, respectively.  In Th\'eor\`eme 3.8 of \cite{Ab14a}, Abdellatif has shown that $\textnormal{ind}_B^G(\widetilde{\chi})$ is reducible if and only if $\widetilde{\chi} = \widetilde{\eta}\circ\det$, in which case we have a nonsplit short exact sequence 
$$0 \rightarrow \widetilde{\eta}\circ\det \rightarrow \textnormal{ind}_B^G(\widetilde{\eta}\circ\det)\rightarrow \widetilde{\eta}\circ\det\otimes \textnormal{St}_G\rightarrow 0,$$ 
where $\textnormal{St}_G = \textnormal{ind}_B^G(1)/1$ is the (irreducible) Steinberg representation of $G$.

The Bruhat decomposition applied to $K$ and the Iwasawa decomposition together imply that $$G = BI \sqcup Bn_sI = BI(1) \sqcup Bn_sI(1).$$  Therefore, we may take as a basis $\textrm{ind}_B^G(\widetilde{\chi})^{I(1)}$ the functions $\{f_1,f_2\}$, defined by 
\begin{center}
 \begin{tabular}{ll}
  $f_1(1) = 1,$ & $f_1(n_s) = 0,$\\
  $f_2(1) = 0,$ & $f_2(n_s) = 1;$
 \end{tabular}
\end{center}
the function $f_1$ is the unique $I(1)$-invariant function with support $BI(1)$ taking the value $1$ at the identity (likewise for $f_2$, supported in $Bn_sI(1)$).

Since $T_1$ is a pro-$p$ subgroup, the restriction of $\widetilde{\chi}$ to $T_1$ must be trivial.  Let $\chi$ denote the representation of $H = T_0/T_1$, given by restricting $\widetilde{\chi}$ to $T_0$.  The action of $\hh_{\ol{\mathbb{F}}_p}(G,I(1))$ on $\textnormal{ind}_B^G(\widetilde{\chi})^{I(1)}$ will depend on the character $\chi$.  If $\chi^s = \chi$, then in the notation of Lemma \ref{decomp} we have 
$$\textrm{ind}_B^G(\widetilde{\chi})^{I(1)} = \textrm{ind}_B^G(\widetilde{\chi})^{I,\chi},$$ 
and the action of $\hh_{\ol{\mathbb{F}}_p}(G,I(1))$ factors through algebra $\hh_{\ol{\mathbb{F}}_p}(G,\chi)$ (via the decomposition of Proposition \ref{orbits}).  Likewise, if $\chi^s\neq\chi$, then 
$$\textrm{ind}_B^G(\widetilde{\chi})^{I(1)} = \textrm{ind}_B^G(\widetilde{\chi})^{I,\chi\oplus\chi^s} = \textrm{ind}_B^G(\widetilde{\chi})^{I,\chi}\oplus\textrm{ind}_B^G(\widetilde{\chi})^{I,\chi^s},$$ 
and the action of $\hh_{\ol{\mathbb{F}}_p}(G,I(1))$ factors through $\hh_{\ol{\mathbb{F}}_p}(G,\chi\oplus\chi^s)$.

\begin{thm}\label{ps} The algebra $\hh_{\ol{\mathbb{F}}_p}(G,I(1))$ acts on $\langle f_1, f_2\rangle_{\ol{\mathbb{F}}_p}$ in the following way:
 \begin{enumerate}[(i)]
  \item If $\chi$ is of trivial type, then
\begin{center}
 \begin{tabular}{cclccclccclcccl}
  $f_1\cdot e_{\chi}$ & = & $f_1,$ & & $f_1\cdot e_{\chi'}$& = & $0,$ & & $f_1\cdot \T_{n_s}$ & = & $\phantom{-}f_2,$ & & $f_1\cdot \T_{n_{s'}}$ & = & $-f_1$\\
  $f_2\cdot e_{\chi}$ & = & $f_2,$ & & $f_2\cdot e_{\chi'}$& = & $0,$ & & $f_2\cdot \T_{n_s}$ & = & $-f_2,$ & & $f_2\cdot \T_{n_{s'}}$ & = & $\widetilde{\chi}(\alpha)f_1,$
 \end{tabular}
\end{center}
for $\chi'\neq \chi$.  
   \item If $\chi$ is hybrid, then 
\begin{center}
 \begin{tabular}{cclccclccclcccl}
  $f_1\cdot e_{\chi}$ & = & $f_1,$ & & $f_1\cdot e_{\chi'}$& = & $0,$ & & $f_1\cdot \T_{n_s}$ & = & $f_2,$ & & $f_1\cdot \T_{n_{s'}}$ & = & $-f_1$\\
  $f_2\cdot e_{\chi}$ & = & $f_2,$ & & $f_2\cdot e_{\chi'}$& = & $0,$ & & $f_2\cdot \T_{n_s}$ & = & $0,$ & & $f_2\cdot \T_{n_{s'}}$ & = & $\widetilde{\chi}(\alpha)f_1,$
 \end{tabular}
\end{center}
for $\chi' \neq \chi$.  
   \item If $\chi$ is regular, then
\begin{center}
 \begin{tabular}{cclccclccclcccc}
  $f_1\cdot e_{\chi}$ & = & $f_1,$ & & $f_1\cdot e_{\chi'}$& = & $0,$ & & $f_1\cdot \T_{n_s}$ & = & $f_2,$ & & $f_1\cdot \T_{n_{s'}}$ & = & $0$\\
  $f_2\cdot e_{\chi^s}$ & = & $f_2,$ & & $f_2\cdot e_{\chi''}$& = & $0,$ & & $f_2\cdot \T_{n_s}$ & = & $0,$ & & $f_2\cdot \T_{n_{s'}}$ & = & $\widetilde{\zeta}(-1)\widetilde{\chi}(\alpha)f_1,$
 \end{tabular}
\end{center}
for $\chi' \neq \chi,~\chi'' \neq \chi^s$.  
 \end{enumerate}
\end{thm}

\begin{proof}
 See Appendix \ref{app}.
\end{proof}

\begin{cor}\label{psmod} In the notation of Definition \ref{append}, the $\hh_{\ol{\mathbb{F}}_p}(G,I(1))$-module $\textnormal{ind}_B^G(\widetilde{\chi})^{I(1)}$ is given by the following:
 \begin{enumerate}[(i)]
  \item Assume $\chi$ is of trivial type.  Then $\textnormal{ind}_B^G(\widetilde{\chi})^{I(1)}\cong M(\widetilde{\chi}(\alpha),\chi)$ as right $\hh_{\ol{\mathbb{F}}_p}(G,I(1))$-modules.
  \item Assume $\chi$ is hybrid.  Then $\textnormal{ind}_B^G(\widetilde{\chi})^{I(1)}\cong M(\widetilde{\chi}(\alpha),\chi)$ as right $\hh_{\ol{\mathbb{F}}_p}(G,I(1))$-modules.
  \item Assume $\chi$ is regular.  Then $\textnormal{ind}_B^G(\widetilde{\chi})^{I(1)}\cong M(0,\widetilde{\chi}(\alpha),\chi)$ as right $\hh_{\ol{\mathbb{F}}_p}(G,I(1))$-modules.
 \end{enumerate}
\end{cor}

\subsection{Supersingular Modules}
In light of the results of the previous section, we make the following definition:

\begin{defn}\label{ssingdef}
 Let $\chi = \zeta\otimes\eta$ be a character of the finite torus $H$.  We define the following characters of $\hh_{\ol{\mathbb{F}}_p}(G,I(1))$:
\begin{enumerate}[(i)]
 \item Assume that $\chi$ is of trivial type.  We set
 \begin{center}
  \begin{tabular}{lcclccclccclcccl}
   $M_{\chi, (S,\emptyset)}$: & $e_\chi$ & $\mapsto$ & $1$, & & $e_{\chi'}$ & $\mapsto$ & $0$, & & $\T_{n_s}$ & $\mapsto$ & $\phantom{-}0$, & & $\T_{n_{s'}}$ & $\mapsto$ & $-1$;\\
   $M_{\chi, (\emptyset,S')}$: & $e_\chi$ & $\mapsto$ & $1$, & & $e_{\chi'}$ & $\mapsto$ & $0$, & & $\T_{n_s}$ & $\mapsto$ & $-1$, & & $\T_{n_{s'}}$ & $\mapsto$ & $\phantom{-}0$,
  \end{tabular}
 \end{center}
for $\chi'\neq \chi$.
 \item Assume that $\chi$ is hybrid.  We set
 \begin{center}
  \begin{tabular}{lcclccclccclcccl}
   $M_{\chi, (\emptyset, S')}$: & $e_\chi$ & $\mapsto$ & $1$, & & $e_{\chi'}$ & $\mapsto$ & $0$, & & $\T_{n_s}$ & $\mapsto$ & $0$, & & $\T_{n_{s'}}$ & $\mapsto$ & $\phantom{-}0$;\\
   $M_{\chi, (\emptyset,\emptyset)}$: & $e_\chi$ & $\mapsto$ & $1$, & & $e_{\chi'}$ & $\mapsto$ & $0$, & & $\T_{n_s}$ & $\mapsto$ & $0$, & & $\T_{n_{s'}}$ & $\mapsto$ & $-1$,
  \end{tabular}
 \end{center}
for $\chi'\neq \chi$.
  \item Assume that $\chi$ is regular.  We set
 \begin{center}
  \begin{tabular}{lcclccclccclcccl}
   $M_{\chi, (\emptyset, \emptyset)}$: & $e_\chi$ & $\mapsto$ & $1$, & & $e_{\chi'}$ & $\mapsto$ & $0$, & & $\T_{n_s}$ & $\mapsto$ & $0$, & & $\T_{n_{s'}}$ & $\mapsto$ & $0$,
  \end{tabular}
 \end{center}
for $\chi'\neq \chi$.  
\end{enumerate}
The modules defined in this way are supersingular (as defined in Definition \ref{ssingdef1}).  We will denote a generic supersingular module above by $M_{\chi,\mathbf{J}}$, where $\mathbf{J} = (J,J')$ is an ordered pair as above with $J\subset J_0(\chi), J'\subset J_0'(\chi)$.  This notation is motivated by the notation of Section \ref{repsandhmods} (cf. Definitions \ref{jandj'} and \ref{ssingmods}).
\end{defn}

Note that we have $M_{\chi,\tj} \cong M_{\chi',\tj'}$ if and only if $\chi = \chi'$ and $\tj = \tj'$.  The computations of the previous sections lead to the following corollary:

\begin{cor}
\begin{enumerate}[(i)]
 \item Let $M$ be a finite-dimensional supersingular $\hh_{\ol{\mathbb{F}}_p}(G,I(1))$-module.  Then $M\cong M_{\chi,\tj}$ for some $\chi$ and $\tj$.
 \item The functor of $I(1)$-invariants induces a bijection between (isomorphism classes of) irreducible nonsupersingular representations of $G$ and (isomorphism classes of) nonsupersingular finite-dimensional simple right $\hh_{\ol{\mathbb{F}}_p}(G,I(1))$-modules.  
 \end{enumerate}\label{maincor}
\end{cor}

\begin{proof}
 This follows from Theorems \ref{thmiw}, \ref{thmiw2}, \ref{thmneqiw}, and Corollary \ref{psmod}.
\end{proof}

\section{Representations of the Finite Hecke Algebras and (Pro)finite Groups}\label{repsandhmods}

\subsection{Finite Hecke Algebras}
We first describe the Hecke algebras for the finite groups $\Gamma = \textnormal{U}(2,1)(\mathbb{F}_{q^2}/\mathbb{F}_{q})$ and $\Gamma' = (\textnormal{U}(1,1)\times\textnormal{U}(1))(\mathbb{F}_{q^2}/\mathbb{F}_{q})$, and their associated simple modules.  

\begin{defn}
We define 
$$\hh_\Gamma := \textrm{End}_\Gamma(\textrm{ind}_\uu^\Gamma(1)),\qquad \hh_{\Gamma'} := \textrm{End}_{\Gamma'}(\textrm{ind}_{\uu'}^{\Gamma'}(1)),$$
where $1$ denotes the trivial character of $\uu$ or $\uu'$.  
\end{defn}

Extending functions by zero induces the injections of $K$- and $K'$-representations (respectively)
$$\textrm{ind}_{\uu}^{\Gamma}(1)\cong \textrm{ind}_{I(1)}^{K}(1)\hookrightarrow \textrm{c-ind}_{I(1)}^G(1)\quad \textnormal{and}\quad \textrm{ind}_{\uu'}^{\Gamma'}(1)\cong \textrm{ind}_{I(1)}^{K'}(1)\hookrightarrow \textrm{c-ind}_{I(1)}^G(1).$$  Passing to $I(1)$-invariants, we view the algebras $\hh_\Gamma$ and $\hh_{\Gamma'}$ as subalgebras of $\hh_{\ol{\mathbb{F}}_p}(G,I(1))$ via
$$\hh_\Gamma \hookrightarrow \textrm{Hom}_K(\textrm{ind}_{I(1)}^{K}(1),\textrm{c-ind}_{I(1)}^G(1)|_K) \cong \hh_{\ol{\mathbb{F}}_p}(G,I(1)),$$
$$\hh_{\Gamma'} \hookrightarrow \textrm{Hom}_{K'}(\textrm{ind}_{I(1)}^{K'}(1),\textrm{c-ind}_{I(1)}^G(1)|_{K'}) \cong \hh_{\ol{\mathbb{F}}_p}(G,I(1)).$$
We deduce from these morphisms that the algebra $\hh_{\Gamma}$ is generated by $\T_{n_s}$ and $e_\chi$ for all characters $\chi$ of $H$, while $\hh_{\Gamma'}$ is generated by $\T_{n_{s'}}$ and $e_\chi$ for all characters $\chi$ of $H$.

\begin{defn}\label{jandj'}
 Let $\chi\in\widehat{H}$ and let $S := \{s\}$ and $S' := \{s'\}$ denote the sets of Coxeter generators for the Weyl groups associated to $\Gamma$ and $\Gamma'$, respectively.  We define
\begin{center}
\begin{tabular}{ccl}
$J_0(\chi)$ & := & $\begin{cases} S & \textrm{if $\chi$ factors through the determinant},\\ \emptyset & \textrm{otherwise}, \end{cases}$ \\
$J_0'(\chi)$ & := & $\begin{cases}S' & \textrm{if} \ \chi^s = \chi,\\ \emptyset & \textrm{otherwise}. \end{cases}$
\end{tabular}
\end{center}
\end{defn}

\begin{defn}\label{ssingmods}
 Let $\chi\in\widehat{H}$.
\begin{enumerate}[(i)]
 \item Let $J\subset J_0(\chi)$, and let $M_{\chi,J}$ denote the character of $\hh_\Gamma$ given by 
\begin{center}
 \begin{tabular}{cclccclcccl}
  $e_\chi$ & $\mapsto$ & $1$, & & $e_{\chi'}$ & $\mapsto$ & $0$, & & $\T_{n_s}$ & $\mapsto$ & $\begin{cases}\phantom{-}0 & \textrm{if}\ s\in J,\\ -1 & \textrm{if}\ s\in J_0(\chi)\smallsetminus J,\\ \phantom{-}0 & \textrm{if}\ s\not\in J_0(\chi).\end{cases}$
 \end{tabular}
\end{center}
for $\chi'\neq \chi$.
 \item Let $J'\subset J_0'(\chi)$, and let $M_{\chi,J'}'$ denote the character of $\hh_{\Gamma'}$ given by 
\begin{center}
 \begin{tabular}{cclccclcccl}
  $e_\chi$ & $\mapsto$ & $1$, & & $e_{\chi'}$ & $\mapsto$ & $0$, & & $\T_{n_{s'}}$ & $\mapsto$ & $\begin{cases}\phantom{-}0 & \textrm{if}\ s'\in J',\\ -1 & \textrm{if}\ s'\in J_0'(\chi)\smallsetminus J',\\ \phantom{-}0 & \textrm{if}\ s'\not\in J_0'(\chi).\end{cases}$
 \end{tabular}
\end{center}
for $\chi'\neq \chi$.
\end{enumerate}
\end{defn}

With these definitions in place, we arrive at the following proposition.

\begin{prop}\label{finheckemods}
Let $\chi\in\widehat{H}$.
 \begin{enumerate}[(i)]
  \item Every simple right $\hh_\Gamma$-module is isomorphic to a character $M_{\chi,J}$ with $J\subset J_0(\chi)$.
  \item Every simple right $\hh_{\Gamma'}$-module is isomorphic to a character $M'_{\chi,J'}$ with $J'\subset J_0'(\chi)$.  
 \end{enumerate}
\end{prop}

\begin{proof}
 The pairs $(\bb,(N\cap K)/(N\cap K_1))$ and $(\bb',(N\cap K')/(N\cap K_1'))$ form ``strongly split BN pairs of characteristic $p$'' (cf. \cite{CE04} Definition 2.20).  The result then follows from Theorem 6.10(iii) of \emph{loc. cit.}.  \end{proof}

\subsection{Carter-Lusztig Theory}  We may now begin to classify the mod-$p$ representations of the finite groups $\Gamma$ and $\Gamma'$.  

\begin{prop}
\begin{enumerate}[(i)]
 \item The  functor $\rho\mapsto \rho^{\uu}$ induces a bijection between isomorphism classes of irreducible representations of $\Gamma$ and isomorphism classes of simple right $\hh_{\Gamma}$-modules.  
 \item The functor $\rho'\mapsto (\rho')^{\uu'}$ induces a bijection between isomorphism classes of irreducible representations of $\Gamma'$ and isomorphism classes of simple right $\hh_{\Gamma'}$-modules. 
\end{enumerate}\label{repsmods}
\end{prop}

\begin{proof}
 As $\hh_\Gamma$ and $\hh_{\Gamma'}$ are Frobenius algebras, this follows from Proposition 1.25(ii) of \cite{CE04}.  
\end{proof}

In light of this Proposition, we make the following definition:

\begin{defn}\label{defoffinreps}
Let $\chi\in\widehat{H}$.  
 \begin{enumerate}[(i)]
  \item For $J\subset J_0(\chi)$, we define $\rho_{\chi,J}$ to be the representation of $\Gamma$ such that $\rho_{\chi,J}^\uu \cong M_{\chi,J}$.
  \item For $J'\subset J_0'(\chi)$, we define $\rho_{\chi,J'}'$ to be the representation of $\Gamma'$ such that $(\rho_{\chi,J'}')^{\uu'} \cong M_{\chi,J'}'$.
 \end{enumerate}
\end{defn}

Given a nonzero irreducible mod-$p$ representation $\rho$ of $\Gamma$, we have $\rho^\uu\neq \{0\}$; Frobenius Reciprocity gives a surjection from $\textnormal{ind}_\uu^\Gamma(1)$ onto $\rho$, where $1$ denotes the trivial character of $\uu$.  Since $\textnormal{ind}_\uu^\Gamma(1)$ decomposes as a direct sum of $\textnormal{ind}_\bb^\Gamma(\chi)$, we see that $\rho$ is actually a quotient of a parabolically induced representation.  In \cite{CL76}, Carter and Lusztig show how to construct irreducible quotients of $\textnormal{ind}_\bb^\Gamma(\chi)$ by using the Hecke operators $e_\chi$ and $\T_{n_s}$ (with everything holding analogously for $\Gamma'$).  These quotients are related to the representations of Definition \ref{defoffinreps} as follows.

\begin{prop}\label{compmodsCL}
Let $\chi\in\widehat{H}$.  
 \begin{enumerate}[(i)]
  \item  If $\chi$ factors through the determinant, then $$\rho_{\chi,S} \cong \textnormal{im}(1 + \T_{n_s}:\textnormal{ind}_\bb^\Gamma(\chi)\rightarrow\textnormal{ind}_\bb^\Gamma(\chi)~),$$ $$\rho_{\chi,\emptyset} \cong \textnormal{im}(\T_{n_s}:\textnormal{ind}_\bb^\Gamma(\chi)\rightarrow\textnormal{ind}_\bb^\Gamma(\chi)~).$$
 If $\chi$ does not factor through the determinant, then $$\rho_{\chi,\emptyset} \cong \textnormal{im}(\T_{n_s}:\textnormal{ind}_\bb^\Gamma(\chi)\rightarrow\textnormal{ind}_\bb^\Gamma(\chi^s)~).$$  
  \item If $\chi^s = \chi$, then $$\rho'_{\chi,S'} \cong \textnormal{im}(1 + \T_{n_{s'}}:\textnormal{ind}_{\bb'}^{\Gamma'}(\chi)\rightarrow\textnormal{ind}_{\bb'}^{\Gamma'}(\chi)~),$$ $$\rho'_{\chi,\emptyset} \cong \textnormal{im}(\T_{n_{s'}}:\textnormal{ind}_{\bb'}^{\Gamma'}(\chi)\rightarrow\textnormal{ind}_{\bb'}^{\Gamma'}(\chi)~).$$
 If $\chi^s \neq \chi$, then $$\rho'_{\chi,\emptyset} \cong \textnormal{im}(\T_{n_{s'}}:\textnormal{ind}_{\bb'}^{\Gamma'}(\chi)\rightarrow\textnormal{ind}_{\bb'}^{\Gamma'}(\chi^s)~).$$  
 \end{enumerate}
\end{prop}

\begin{proof}
 Theorem 7.1 and Corollary 7.5 of \cite{CL76} imply that the images of the Hecke operators are irreducible and inequivalent; it therefore suffices to match the two sets of representations.  Theorem 7.1 and Proposition 6.6 of \emph{loc. cit.} give the action of $\hh_\Gamma$ and $\hh_{\Gamma'}$ on the $\uu$- and $\uu'$-invariants of the image representations.  The claim then follows from Proposition \ref{finheckemods} and Definition \ref{defoffinreps}.  
\end{proof}

We record one final result regarding the constituents of $\textnormal{ind}_{\bb'}^{\Gamma'}(\chi)$, which will be of use later.

\begin{lemma}\label{dimlemma} Let $\chi\in\widehat{H}$.  
\begin{enumerate}
 \item Assume $\chi^s = \chi$.  Then $e_\chi(1 + \T_{n_{s'}})e_\chi$ and $-e_\chi \T_{n_{s'}}e_\chi$ are orthogonal idempotents, and induce a splitting $$\textnormal{ind}_{\bb'}^{\Gamma'}(\chi)\cong \rho'_{\chi,S'}\oplus\rho'_{\chi,\emptyset}.$$  
Here $\rho'_{\chi,S'}$ is one-dimensional, and $\rho'_{\chi,\emptyset}$ is a twist of the Steinberg representation of $\Gamma'$. 
 \item Assume $\chi^s \neq \chi$.  Then the complex 
$$0 \rightarrow \rho_{\chi^s,\emptyset}' \rightarrow \textnormal{ind}_{\bb'}^{\Gamma'}(\chi) \stackrel{\T_{n_{s'}}}{\rightarrow} \rho_{\chi,\emptyset}' \rightarrow 0$$
 is exact if and only if $q = p$.  In this case, the sequence is nonsplit.  
\end{enumerate}
\end{lemma}

\begin{proof}
 Part (i) follows exactly as in Lemma 3.7 of \cite{Pas04}.  For part (ii), we use an alternate classification of irreducible representations of the group $\Gamma' \cong (\textnormal{U}(1,1)\times \textnormal{U}(1))(\mathbb{F}_{q^2}/\mathbb{F}_q)$ in terms of highest weight modules (cf. \cite{He09}, \cite{Hum76} or \cite{Ste68}). In this description, we have 
$$\rho_{\chi,\emptyset}'\cong \textrm{Sym}^{j_0}(\ol{\mathbb{F}}_p^2)\otimes\cdots\otimes\textrm{Sym}^{j_{f-1}}(\ol{\mathbb{F}}_p^2)^{\textrm{Fr}^{f-1}}\otimes(\sideset{}{^\star}\det)^k\otimes \omega^c,$$ 
for some $0\leq j_i \leq p - 1$, $0\leq k,c  < q + 1$, where $\det^\star$ denotes the determinant map of $\textnormal{U}(1,1)(\mathbb{F}_{q^2}/\mathbb{F}_q)$, and where $\omega:\textnormal{U}(1)(\mathbb{F}_{q^2}/\mathbb{F}_q)\hookrightarrow \mathbb{F}_{q^2}^\times\stackrel{\iota}{\rightarrow}\ol{\mathbb{F}}_p^\times$.  Consequently, this gives 
$$\rho_{\chi^s,\emptyset}'\cong \textrm{Sym}^{p - 1 - j_0}(\ol{\mathbb{F}}_p^2)\otimes\cdots\otimes\textrm{Sym}^{p - 1 - j_{f-1}}(\ol{\mathbb{F}}_p^2)^{\textrm{Fr}^{f-1}}\otimes(\sideset{}{^\star}\det)^{1 + k + \sum_{i = 0}^{f - 1}j_ip^i}\otimes \omega^c.$$ 
Hence, 
$$\dim_{\ol{\mathbb{F}}_p}(\rho_{\chi,\emptyset}') + \dim_{\ol{\mathbb{F}}_p}(\rho_{\chi^s,\emptyset}') = \prod_{i = 0}^{f-1}(j_i + 1) + \prod_{i = 0}^{f-1}(p - j_i);$$  
this quantity is equal to $q + 1 = \dim_{\ol{\mathbb{F}}_p}(\textrm{ind}_{\bb'}^{\Gamma'}(\chi))$ if and only if $q = p$.  Theorem 7.4 of \cite{CL76} implies that for $q = p$, the exact sequence is nonsplit.
\end{proof}

\subsection{Injective Envelopes}
We begin with some general remarks.  Given any irreducible mod-$p$ representation $\rho$ of $\Gamma$, we may view it as a representation of $K$ via the projection $K\twoheadrightarrow K/K_1\cong \Gamma$.  Conversely, any smooth irreducible representation of $K$ must be of this form; this follows from Lemma 3(1) of \cite{BL94} and the fact that $K_1$ is a normal pro-$p$ subgroup of $K$.  In light of this, we shall abuse notation and henceforth identify smooth irreducible representations of $K$ and those of $\Gamma$ (and analogously for $K'$ and $\Gamma'$, and $I$ and $H$).  

We now briefly recall some results regarding socles and injective envelopes (see \cite{Se77} and \cite{Pas04} for details and definitions).  Let $\mathcal{K}$ be any finite or profinite group, and let $\rho$ be a smooth $\ol{\mathbb{F}}_p$-representation of $\mathcal{K}$.  We let $\textnormal{inj}_{\mathcal{K}}(\rho)$ and $\textnormal{soc}_{\mathcal{K}}(\rho)$ denote the injective envelope and socle of $\rho$, respectively, in the category of smooth representations of $\mathcal{K}$.  

%
%
%
%
%

\begin{lemma}
\begin{enumerate}[(i)]
\item Let $\rho$ be an irreducible representation of $K$ and let $\rho\hookrightarrow\textnormal{inj}_K(\rho)$ be an injective envelope of $\rho$.  Then
\begin{equation}\label{decom for K}
\textnormal{inj}_K(\rho)|_I \cong \bigoplus_{\chi\in \widehat{H}} \textnormal{inj}_I(\chi)^{\oplus m_{\rho,\chi}},
\end{equation}
where $m_{\rho,\chi} = \dim_{\ol{\mathbb{F}}_p}(\textnormal{Hom}_H (\chi, \textnormal{inj}_\Gamma (\rho)^{\mathbb{U}}))$.
 
\item Let $\rho'$ be an irreducible representation of $K'$ and let $\rho'\hookrightarrow\textnormal{inj}_{K'}(\rho')$ be an injective envelope of $\rho'$.  Then
\begin{equation}\label{decom for K'}
\textnormal{inj}_{K'}(\rho')|_I \cong \bigoplus_{\chi\in \widehat{H}} \textnormal{inj}_I(\chi)^{\oplus m_{\rho',\chi}},
\end{equation}
where $m_{\rho',\chi} = \dim_{\ol{\mathbb{F}}_p}(\textnormal{Hom}_H (\chi, \textnormal{inj}_{\Gamma'} (\rho')^{\mathbb{U}'}))$.
\end{enumerate}\label{decomps}

In particular, the integers $m_{\rho,\chi}$ and $m_{\rho',\chi}$ are finite for every character $\chi$ of $H$.
\end{lemma}

\begin{proof}
 The proof is identical to the proof of Lemma 6.19 of \cite{Pas04}.  
\end{proof}

\begin{cor}\label{finiteness of I_1 of inj}
Let $\mathcal{K}\in \{K,~ K'\}$, and let $\rho$ be a smooth representation of $\mathcal{K}$ such that $\textnormal{soc}_{\mathcal{K}}(\rho)$ is of finite length as a $\mathcal{K}$-representation.  Then the space of $I(1)$-invariants of $\textnormal{inj}_{\mathcal{K}}(\rho)$ is finite-dimensional and $\rho$ is admissible.  
\end{cor}

\begin{proof}
We have
$$\rho^{I(1)}\hookrightarrow \textnormal{inj}_{\mathcal{K}}(\rho)^{I(1)} \cong \textnormal{inj}_{\mathcal{K}}(\textnormal{soc}_{\mathcal{K}}(\rho))^{I(1)};$$
finite-dimensionality follows from $\textnormal{inj}_{I}(\chi)^{I(1)}\cong \chi$ and Lemma \ref{decomps}, while admissibility follows from \cite{Pas04}, Lemma 6.18.
\end{proof}

In general, it is not clear how the multiplicities in Lemma \ref{decomps} are related, although we may record one result in this direction:

\begin{lemma}\label{Karol-7}
We have $m_{\rho, \chi}=m_{\rho, \chi^s}$, and $m_{\rho', \chi}=m_{\rho', \chi^s}.$
\end{lemma}
\begin{proof}
The definition of the numbers $m_{\rho,\chi}$ and Frobenius Reciprocity give
\begin{eqnarray*}
 m_{\rho,\chi} & = & \dim_{\ol{\mathbb{F}}_p}(\textnormal{Hom}_H (\chi, \textnormal{inj}_\Gamma(\rho)^{\uu}))\\
 & = & \dim_{\ol{\mathbb{F}}_p}(\textnormal{Hom}_\bb(\chi,\textnormal{inj}_\Gamma(\rho)|_\bb))\\
 & = & \dim_{\ol{\mathbb{F}}_p}(\textnormal{Hom}_\Gamma(\textnormal{ind}_\bb^\Gamma(\chi),\textnormal{inj}_\Gamma(\rho))),
\end{eqnarray*}
which is precisely the multiplicity with which $\rho$ occurs as a composition factor of $\textnormal{ind}_\bb^\Gamma(\chi)$.  The discussion in Sections 7.2 and 9.7 of \cite{Hum06} implies that $\textnormal{ind}_\bb^\Gamma(\chi)$ and $\textnormal{ind}_\bb^\Gamma(\chi^s)$ have the same composition factors, and the result follows.  The proof for $\Gamma'$ is identical.  
\end{proof}

\begin{cor}\label{injdecomp}
Assume $q = p$.  We then have
\begin{center}
\begin{tabular}{ccll}
$\textnormal{inj}_{K'}(\rho'_{\chi, J'})|_I$ & $\cong$ & $\textnormal{inj}_I(\chi)$ & if $\chi^s=\chi$ and $J'\subset J'_0(\chi) = S'$,\\
$\textnormal{inj}_{K'}(\rho'_{\chi, \emptyset})|_I$ & $\cong$ & $\textnormal{inj}_I(\chi)\oplus \textnormal{inj}_I(\chi^s)$ & if $\chi^s\neq \chi$.
\end{tabular}
\end{center}
\end{cor}

\begin{proof}
 This follows from the decomposition \eqref{decom for K'}, the definition of the integers $m_{\rho', \chi}$ used in Lemma \ref{Karol-7}, and the description of the composition factors of $\textnormal{ind}_{\bb'}^{\Gamma'}(\chi)$ in Lemma \ref{dimlemma}.  
\end{proof}

\section{Diagrams and Coefficient Systems}\label{diagsandcoeffs}

In this section, we follow \cite{Pas04} closely and translate the language of coefficient systems and diagrams to the group $G$.  Our case is even easier to some extent, due mainly to the fact that the extended Bruhat-Tits building of $G$ coincides with the reduced Bruhat-Tits building.  We refer to \cite{Ti79} throughout for definitions.  

\subsection{Preliminaries}

Let $X$ be the reduced Bruhat-Tits building of $G$.  The building $X$ is a simplicial complex of dimension 1 (that is, a tree) on which $G$ acts.  We denote by $A$ the apartment corresponding to the maximal $F$-split subtorus of $T$.  There exist neighboring vertices $\sigma_0$ and $\sigma_0'$ in $A$ such that $\textnormal{stab}_G(\sigma_0) = K$ and $\textnormal{stab}_G(\sigma_0') = K'$ (cf. \cite{Ti79} Section 3.1.1).  The vertex $\sigma_0'$ has $q + 1$ neighboring vertices in $X$; the vertex $\sigma_0$ has $q^3 + 1$ neighboring vertices and is hyperspecial (these facts follow from Statement 3.5.4 in \emph{loc. cit.} and $|\Gamma/\bb| = q^3 + 1, |\Gamma'/\bb'| = q+1$).  We let $\tau_1$ denote the edge from $\sigma_0$ to $\sigma_0'$; we have $\textnormal{stab}_G(\tau_1) = I$.

We let $X_i, i \in \{0,1\},$ denote the set of all $i$-dimensional simplices of $X$.  The tree $X$ has a natural $G$-invariant distance function, and we denote by $X_0^e$ (resp. $X_0^o$) the set of vertices at an even (resp. odd) distance from $\sigma_0$.  One easily shows that $X_0^e$ and $X_0^o$ constitute two disjoint orbits for the action of $G$ on $X_0$, while $G$ acts transitively on $X_1$.  

Coefficient systems over $\mathbb{C}$ were first introduced in \cite{SS97} by Schneider and Stuhler, and used in the mod-$p$ setting by Pa\v{s}k\={u}nas in \cite{Pas04}.  We refer to \emph{loc. cit.} for the definitions of coefficient systems, their homology, and elementary properties (note that the results we require from \emph{loc. cit.}, Sections 5.1--5.3, hold in a more general context, and in particular for the group $G$).

\subsection{Diagrams}

We now discuss diagrams, which are simpler and easier to handle than coefficient systems.  

\begin{defn}\label{diag}
A \emph{diagram} is a quintuple $D=(D_{0}, D'_{0}, D_1, r, r')$, in which $(\rho_0, D_0)$ is a smooth representation of $K$, $(\rho'_0, D'_0)$ is a smooth representation of $K'$, $(\rho_1, D_1)$ is a smooth representation of $I$, and $r\in \textnormal{Hom}_{I}(D_1, D_0|_I)$, $r'\in \textnormal{Hom}_{I}(D_1, D'_0|_I)$.
\end{defn}

We represent a diagram pictorially as:

\centerline{
\xymatrix{
 & D_0 \\
D_1 \ar[ur]^{r}\ar[dr]_{r'}& \\
 & D_0'
}
}

\begin{defn}
A \emph{morphism} $\psi$ between two diagrams $D=(D_0, D'_{0}, D_1, r_D, r'_D)$ and $E=(E_0, E'_0, E_1, r_E, r'_E)$ is a triple $(\psi_{0}, \psi'_{0}, \psi_1)$, where $\psi_0\in \textnormal{Hom}_{K}(D_0,E_0)$, $\psi'_0\in \textnormal{Hom}_{K'}(D'_0,E'_0)$, and $\psi_1\in\textnormal{Hom}_{I}(D_1,E_1)$, such that the squares in the following diagram commute as $I$-representations:

\centerline{
\xymatrix{
 & D_0 \ar[rr]^<<<<<<<<<{\psi_0}& & E_0\\
D_1 \ar[ur]^{r_D}\ar[dr]_{r_D'}\ar[rr]^{\psi_1}& & E_1\ar[ur]^{r_E}\ar[dr]_{r_E'} &  \\
 & D_0'\ar[rr]^<<<<<<<<<{\psi_0'} & & E_0'
}
}
\noindent We say $\psi$ is an \emph{embedding} if the maps $\psi_0, \psi'_0,$ and $\psi_1$ are injective.  
\end{defn}

The category of diagrams with morphisms defined above is denoted $\mathcal{DIAG}$. 

\subsection{The Functors $\mathcal{C}$ and $\mathcal{D}$}\label{candd}
We let $\mathcal{COEF}_G$ denote the category of $G$-equivariant coefficient systems on $X$ (cf. \cite{Pas04}, Section 5.1), and let $\mathcal{V} = (V_\sigma)_{\sigma\in X_0\cup X_1}\in\mathcal{COEF}_G$.  In order to discuss the equivalence of $\mathcal{COEF}_G$ and $\mathcal{DIAG}$, we first observe that there is an obvious ``forgetful'' functor in one direction:

\begin{defn}\label{functor one}

Let $\mathcal{D}$ be the functor from $\mathcal{COEF}_{G}$ to $\mathcal{DIAG}$ given by:
\begin{center}
\begin{tabular}{rccl}
$\mathcal{D}:$ & $\mathcal{COEF}_G$ & $\rightarrow$ & $\mathcal{DIAG}$\\
 & $\mathcal{V}=(V_{\sigma})_{\sigma}$ & $\mapsto$ & $(V_{\sigma_0}, V_{\sigma'_0}, V_{\tau_1}, r_{\sigma_0}^{\tau_1}, r_{\sigma'_0}^{\tau_1})$\\
  & & & \xymatrix{
 & V_{\sigma_0} \\
 =~ V_{\tau_1} \ar[ur]^{r_{\sigma_0}^{\tau_1}}\ar[dr]_{r_{\sigma_0'}^{\tau_1}}& \\
 & V_{\sigma_0'}
}
\end{tabular}
\end{center}
\end{defn}

We now recall the construction of the functor $\mathcal{C}$ from \emph{loc. cit.}.  Let $D= (D_{0}, D'_{0}, D_1, r, r')$ denote a diagram in $\mathcal{DIAG}$, fixed throughout the whole discussion.    

\subsubsection{Representations}\label{under space}

We will construct the vector spaces comprising $\mathcal{C}(D)$ as subrepresentations of the following compactly induced representations:
\begin{center}
$\text{c-ind}_{K}^{G}(\rho_0), ~\text{c-ind}_{K'}^{G}(\rho'_0), ~\text{c-ind}_{I}^{G}(\rho_1)$.
\end{center}

For a vertex $\sigma\in X_0$, there exists $g\in G$ such that $\sigma=g.\sigma_0$ or $\sigma = g.\sigma'_0$, depending on whether $\sigma\in X_0^e$ or $\sigma\in X_0^o$.  We define
\begin{center}
$F_\sigma: =$
$\begin{cases}
\{f\in \text{c-ind}_{K}^{G}(\rho_0): ~\text{supp}(f) \subset Kg^{-1}\} & \text{if}~\sigma\in X_0^e,\\

\{f\in \text{c-ind}_{K'}^{G}(\rho'_0): ~\text{supp}(f) \subset K'g^{-1}\} & \text{if}~\sigma\in X_0^o.
\end{cases}$
\end{center}
For an edge $\tau\in X_1$, there exists $g\in G$ such that $\tau=g.\tau_1$.  We define

\begin{center}
$F_{\tau} := \{f\in \text{c-ind}_{I}^{G}(\rho_1): ~\text{supp}(f)\subset Ig^{-1}\}.$
\end{center}
We note that these definitions are independent of the choice of $g$.  

Given the simplex $\sigma_0$ (resp. $\sigma_0'$, resp. $\tau_1$) and a vector $v\in D_0$ (resp. $v\in D_0'$, resp. $v\in D_1$), we will denote by $f_v$ the unique function in $F_{\sigma_0}$ (resp. $F_{\sigma_0'}$, resp. $F_{\tau_1}$) satisfying $\textnormal{supp}(f_v) = K$ (resp. $\textnormal{supp}(f_v) = K'$, resp. $\textnormal{supp}(f_v) = I$) and $f_v(1) = v$.  By construction, given any simplex $\sigma\subset X$ and $f\in F_\sigma$, there exists a $v\in D_0, D_0'$ or $D_1$ and $g\in G$ such that $f = g.f_v$.

Since each of the vector spaces $F_\sigma$ are contained in a $G$-representation, there is an obvious way to equip $\mathcal{F} = (F_\sigma)_\sigma$ with a $G$-action, which is easily seen to satisfy the first two points of Definition 5.3 of \emph{loc. cit.}.

\subsubsection{$G$-equivariant restriction maps}\label{restriction maps}

We first define the restriction maps $r^{\tau_1}_{\sigma_0}$ and $r^{\tau_1}_{\sigma_0'}$ by
$$r_{\sigma_0}^{\tau_1}(f_v) := f_{r(v)},\qquad r_{\sigma_0'}^{\tau_1}(f_v) := f_{r'(v)},$$
where $v\in D_1$.  The maps $r_{\sigma_0}^{\tau_1}$ and $r_{\sigma_0'}^{\tau_1}$ are well-defined, and are easily seen to be $I$-equivariant.

Now let $\tau = \{\sigma,\sigma'\}$ be an edge such that $\sigma\in X_0^e$, $\sigma'\in X_0^o$.  There exists an element $g$ which satisfies $\tau=g.\tau_{1}$ (which implies $\sigma=g.\sigma_0$ and $\sigma' = g.\sigma_0'$); such a choice of $g$ is unique up to an element of $I$.  Since $g$ defines a vector space isomorphism between $F_{\tau_1}$ and $F_\tau$, every element of $F_\tau$ is of the form $g.f_v$, for some $v\in D_1$.  We define $r^\tau_\sigma$ and $r^\tau_{\sigma'}$ in the unique way which makes the diagram in the third point of Definition 5.3 of \emph{loc. cit.} hold.  Explicitly,  
$$r_{\sigma}^{\tau}(g.f_v) = r_{g.\sigma_0}^{g.\tau_1}(g.f_v) := g.r_{\sigma_0}^{\tau_1}(f_v) = g.f_{r(v)},\qquad r_{\sigma'}^{\tau}(g.f_v) = r_{g.\sigma_0'}^{g.\tau_1}(g.f_v) := g.r_{\sigma_0'}^{\tau_1}(f_v) = g.f_{r'(v)}.$$

\subsubsection{Morphisms}\label{morphisms}

Let $D=(D_0, D'_{0}, D_1, r_D, r'_D)$ and $E=(E_0, E'_0, E_1, r_E, r'_E)$ be two diagrams, and $\psi=(\psi_{0}, \psi'_{0}, \psi_1)$ a morphism between them.  Let $\mathcal{F}=(F_\sigma)_\sigma$ and $\mathcal{F}'=(F'_\sigma)_\sigma$ be the coefficient systems associated to $D$ and $E$, respectively, constructed above.  

Let $\sigma\in X_0^e$, let $g\in G$ be such that $\sigma=g.\sigma_0$, and let $v\in D_0$.  Given $g.f_v\in F_\sigma$, we define the element $\psi_\sigma(g.f_v)\in F_\sigma'$ by
$$\psi_{\sigma}(g.f_v) := g.f_{\psi_0(v)}.$$
Likewise, we define 
$$\psi_{\sigma'}(g.f_v) := g.f_{\psi_0'(v)}$$ 
if $\sigma'\in X_0^o$, $\sigma' = g.\sigma_0'$, and $v\in D_0'$.

Let $\tau$ be an edge, let $g\in G$ be such that $\tau=g.\tau_1$, and let $v\in D_1$.  Given $g.f_v\in F_\tau$, we define the $\psi_\tau(g.f_v)\in F_\tau'$ by
$$\psi_{\tau}(g.f_v) := g.f_{\psi_1(v)}.$$
Note that these maps are well-defined.  

Finally, we must verify that these linear maps $(\psi_{\sigma})_\sigma$ are compatible with the restriction maps and the action of $G$.  Both claims follow directly from the definitions.  

We may now make the following definition:
\begin{defn}\label{defofc}
Let $\mathcal{C}$ be the map:
\begin{center}
\begin{tabular}{lrcl}
$\mathcal{C}:$ &  $\mathcal{DIAG}$ & $\rightarrow$ &  $\mathcal{COEF}_{G}$\\
 & $D = (D_0,D_0',D_1,r,r')$ & $\mapsto$ & $\mathcal{F} = (F_{\sigma})_\sigma,$
\end{tabular}
\end{center}
where $(F_\sigma)_\sigma$ is the coefficient system defined above.
\end{defn}

The results of the previous subsections imply that $\mathcal{C}$ is a bona fide functor between the two categories. 

\begin{thm}\label{equivalence}
The categories $\mathcal{DIAG}$ and $\mathcal{COEF}_{G}$ are equivalent.  The equivalence is induced by the functors 
\begin{center}
\begin{tabular}{rccc}
$\mathcal{D}:$ & $\mathcal{COEF}_G$ & $\rightarrow$ & $\mathcal{DIAG}$\\
$\mathcal{C}:$ & $\mathcal{DIAG}$ & $\rightarrow$ & $\mathcal{COEF}_G$,
\end{tabular}
\end{center}
where $\mathcal{C}$ and $\mathcal{D}$ are defined in Definitions \ref{functor one} and \ref{defofc} below.
\end{thm}

\begin{proof}
 The proof is identical to the proof of Theorem 5.17 in \emph{loc. cit.}.
\end{proof}

\section{Supersingular Representations}\label{init}

\subsection{Initial Diagrams}\label{initial diag}

Using the functor $\mathcal{C}$, we may now construct coefficient systems by defining the appropriate diagrams.  In particular, to each supersingular $\hh_{\ol{\mathbb{F}}_p}(G,I(1))$-module we associate a diagram as follows.  

\begin{defn}\label{diagram one}
Let $\chi\in\widehat{H}$, and let $M_{\chi,\tj}$ be a supersingular $\hh_{\ol{\mathbb{F}}_p}(G,I(1))$-module, with $\tj = (J,J')$ as in Definition \ref{ssingdef}.  We associate to $M_{\chi,\tj}$ the diagram
$$\xymatrix{
 & &  \rho_{\chi,J} \\
 D_{\chi,\tj}:= \left(\rho_{\chi, J},~ \rho'_{\chi, J'},~ \chi, ~j, ~ j'\right) = & \chi \ar[ur]^{j}\ar[dr]_{j'}& \\
 & &  \rho'_{\chi,J'}
}$$
where $j$ and $j'$ are inclusion maps, and define 
$$\mathcal{D}_{\chi, \tj} := \mathcal{C}(D_{\chi, \tj})$$
to be the associated $G$-equivariant coefficient system.  We let the underlying space of the $I$-representation $\chi$ be spanned by a fixed vector, which we identify with its image in $\rho_{\chi,J}^{I(1)} = \rho_{\chi,J}^\uu$ and $(\rho_{\chi,J'}')^{I(1)} = (\rho_{\chi,J'}')^{\uu'}$ via $j$ and $j'$.  
\end{defn}

\begin{remark}
We note that if $M_{\chi,\tj}$ and $D_{\chi,\tj}$ are as above, we have $\rho_{\chi,J}^\uu\cong M_{\chi,\tj}|_{\hh_\Gamma}$ as $\hh_\Gamma$-modules and $(\rho_{\chi,J'}')^{\uu'}\cong M_{\chi,\tj}|_{\hh_{\Gamma'}}$ as $\hh_{\Gamma'}$-modules.  
\end{remark}

\begin{prop}\label{reduced to the irr quotient of H_0}
Let $M_{\chi, \tj}$ be a supersingular $\hh_{\ol{\mathbb{F}}_p}(G,I(1))$-module, and let $\pi$ be a nonzero irreducible quotient of $H_0(X,\mathcal{D}_{\chi,\tj})$, the $0$-homology of the coefficient system $\mathcal{D}_{\chi,\tj}$ (cf. \cite{Pas04}, Section 5.2).  Then $\pi^{I(1)}$ contains $M_{\chi,\tj}$, and $\pi$ is supersingular as a $G$-representation.
\end{prop}

\begin{proof}
In the notation of Subsection \ref{under space} and Section 5.2 of \emph{loc. cit.}, we let $\omega_{\sigma_0, f_v}\in C_c^{or}(X_{(0)},\mathcal{D}_{\chi,\tj})$ be the $0$-chain supported on $\sigma_0$, satisfying $\omega_{\sigma_0, f_v}(\sigma_0) = f_v$, where $v$ is a fixed vector spanning the underlying space of $\chi$.  By definition of the $G$-action, $\omega_{\sigma_0, f_v}$ is $I(1)$-invariant and the group $I$ acts by the character $\chi$.  Let $\bar{\omega}_{\sigma_0, f_v}$ denote its image in $H_0(X,\mathcal{D}_{\chi,\tj})$.  To proceed, we must show two things:

\begin{enumerate}[(i)]
\item The element $\bar{\omega}_{\sigma_0, f_v}$ generates $H_0(X, \mathcal{D}_{\chi, \tj})$ as a $G$-representation.
\item The right action of $\mathcal{H}_{\ol{\mathbb{F}}_p}(G, I(1))$ on $\langle \bar{\omega}_{\sigma_0, f_v} \rangle_{\ol{\mathbb{F}}_p}$ yields an isomorphism onto $M_{\chi, \tj}$.
\end{enumerate}

Assuming these two claims, we let $\pi$ be a nonzero irreducible quotient of $H_0(X, \mathcal{D}_{\chi, \tj})$.  Since $\bar{\omega}_{\sigma_0,f_v}$ generates $H_0(X,\mathcal{D}_{\chi, \tj})$, its image in $\pi$ will be nonzero.  The second result above then shows that $\pi^{I(1)}$ contains the $\mathcal{H}_{\ol{\mathbb{F}}_p}(G, I(1))$-module $M_{\chi, \tj}$ and the proposition follows from Corollary \ref{maincor}.

It remains to prove the two claims.  For the first, we note that if $\omega_{\sigma_0', f_v}\in C_c^{or}(X_{(0)},\mathcal{D}_{\chi,\tj})$ denotes the $0$-chain supported on $\sigma_0'$ satisfying $\omega_{\sigma_0', f_v}(\sigma_0') = f_v$, then Lemma 5.6 of \emph{loc. cit.} implies $\bar{\omega}_{\sigma_0, f_v} = \bar{\omega}_{\sigma_0', f_v}$ in $H_0(X,\mathcal{D}_{\chi, \tj})$.  Since any irreducible representation of $K$ or $K'$ is generated by its space of $I(1)$-invariants, $\bar{\omega}_{\sigma_0, f_v}$ (resp. $\bar{\omega}_{\sigma_0', f_v}$) generates the image in $H_0(X,\mathcal{D}_{\chi, \tj})$ of the space $C_c^{or}(\sigma_0,\mathcal{D}_{\chi, \tj})$ (resp. $C_c^{or}(\sigma_0',\mathcal{D}_{\chi, \tj})$) of $0$-chains supported on $\sigma_0$ (resp. $\sigma_0'$).  This fact, combined with the observation that $G$ acts transitively on the sets $X_0^e$ and $X_0^o$, verifies the claim.  

For the second claim, note that by Definition \ref{defoffinreps} and our choice of irreducible $K$- and $K'$-representations, we have
\begin{center}
\begin{tabular}{ccccc}
$\langle v\rangle_{\ol{\mathbb{F}}_p}$ & $=$ & $(\rho_{\chi, J})^{I(1)}$ & $\cong$ &  $M_{\chi, J}$ as $\mathcal{H}_\Gamma$-modules,\\
$\langle v\rangle_{\ol{\mathbb{F}}_p}$ & $=$ & $(\rho'_{\chi, J'})^{I(1)}$ & $\cong$ &  $M'_{\chi, J'}$ as $\mathcal{H}_{\Gamma'}$-modules,
\end{tabular}
\end{center}
where $\tj = (J,J')$.  We conclude from Theorem \ref{propstr} that $\langle \bar{\omega}_{\sigma_0, f_v} \rangle_{\ol{\mathbb{F}}_p}$ is equivalent to $M_{\chi, \tj}$ as a right $\mathcal{H}(G, I(1))$-module.
\end{proof}

\subsection{Pure Diagrams}\label{advanced}

In light of Proposition \ref{reduced to the irr quotient of H_0}, it suffices to construct irreducible quotients of $H_0(X,\mathcal{D}_{\chi, \tj})$ to produce supersingular representations.  With this in mind, we make the following definition:

\begin{defn}\label{karol-2}
Let $M_{\chi,\tj}$ be a supersingular module, and let ${D}=(D_0, D'_0, D_1, r, r')$ be a diagram.  We say $D$ is \emph{pure with respect to} $M_{\chi, \tj}$ if it satisfies the following conditions:
\begin{enumerate}[(i)]
\item There exists an embedding of diagrams:
$$\psi: {D}_{\chi, \tj}\hookrightarrow D.$$
\item The maps $r$ and $r'$ induce isomorphisms $D_0|_I \cong D'_0|_I \cong D_1 $.
\item Either $\textnormal{soc}_K(D_0)$ or $\textnormal{soc}_{K'}(D'_0)$ is irreducible.
\end{enumerate}
\end{defn}

With these definitions, we prove the following result, whose proof is due to Pa\v{s}k\={u}nas (\cite{Pas04}).  

\begin{thm}\label{Karol-3}
Let $M_{\chi, \tj}$ be a supersingular module, and suppose that $D$ is a pure diagram with respect to $M_{\chi,\tj}$. Then the image of the induced $G$-morphism between the $0$-homology
\begin{center}
$\pi_{{D}} =\textnormal{im}(\psi_*:H_0(X, \mathcal{D}_{\chi, \tj})\rightarrow H_0(X, \mathcal{C}({D}))~)$
\end{center}
is irreducible, admissible and supersingular.   
\end{thm}

\begin{proof}
Let $D = (D_0, D'_0, D_1, r, r')$.  To verify the result, it suffices to show $\pi_{{D}}$ is irreducible, admissible and nonzero, by Proposition \ref{reduced to the irr quotient of H_0}.  Let us assume that $\textnormal{soc}_K(D_0)$ is irreducible; the case with $\textnormal{soc}_{K'}(D_0')$ irreducible is the same.    

Since $\psi$ is an embedding, we have $\psi_*(\bar{\omega}_{\sigma_0,f_v})\neq 0$ (so $\pi_D \neq \{0\}$), and therefore the sub-$K$-representation it generates is irreducible (by definition of $\mathcal{D}_{\chi,\tj}$).  By Proposition 5.10 of \cite{Pas04}, we obtain
$$\{0\}\neq \langle K.\psi_*(\bar{\omega}_{\sigma_0,f_v})\rangle_{\ol{\mathbb{F}}_p} \subset \textnormal{soc}_K(H_0(X,\mathcal{C}(D))|_K) \cong \textnormal{soc}_K(D_0),$$
and therefore the inclusion is an equality.  

Now let $\pi'$ be a nonzero $G$-invariant subspace of $\pi_D$.  Since $\textnormal{soc}_K(\pi'|_K)\neq \{0\}$, we have
$$\{0\}\neq \textnormal{soc}_K(\pi'|_K)^{I(1)}\subset \textnormal{soc}_K(H_0 (X, \mathcal{C}({D}))|_K)^{I(1)} = \langle K.\psi_*(\bar{\omega}_{\sigma_0,f_v})\rangle_{\ol{\mathbb{F}}_p}^{I(1)} = \langle \psi_*(\bar{\omega}_{\sigma_0,f_v})\rangle_{\ol{\mathbb{F}}_p},$$
and therefore the inclusion must be an equality.  This shows $\psi_*(\bar{\omega}_{\sigma_0,f_v})\in \pi'$, and since this vector generates $\pi_D$, we must have $\pi' = \pi_D$.

To show admissibility, we observe that $$\pi_D|_K\subset H_0(X,\mathcal{C}(D))|_K \cong D_0,$$ which implies that $\textnormal{soc}_K(\pi_D|_K)$ is simple.  The claim then follows from Corollary \ref{finiteness of I_1 of inj}.
\end{proof}

The definitions of pure and essentially pure diagrams do not make it clear that such diagrams exist in general.  We take up this question when $q = p$ in the next section.

\subsection{Construction of Pure Diagrams when $q = p$}

We now give an application of the formalism developed in the previous section, using results of Section \ref{repsandhmods}.  

\begin{thm}\label{Karol}
Suppose $q=p$. Then for every supersingular module $M_{\chi, \tj},$ there exists a pure diagram with respect to $M_{\chi, \tj}$.  More precisely, the corresponding initial diagram
\begin{center}
$D_{\chi, \tj}=(\rho_{\chi,J},~ \rho'_{\chi,J'},~ \chi, ~j, ~ j')$
\end{center}
can be embedded into a pure diagram
\begin{center}
${E}_{\chi, \tj}= (\textnormal{inj}_K(\textnormal{P}), ~\textnormal{inj}_{K'}(\textnormal{P}'), ~\textnormal{inj}_I(\textnormal{X}),~ \boldsymbol{j},~\boldsymbol{j}')$
\end{center}
where $\textnormal{P} = \rho_{\chi,J}$, $\textnormal{P}'$ is a semisimple representation of $K'$ having $\rho'_{\chi,J'}$ as a summand, and $\textnormal{X}$ is a semisimple representation of $I$ having $\chi$ as a summand.  
\end{thm}

\begin{proof}
Using Lemma \ref{Karol-7}, we rewrite equation $\eqref{decom for K}$ as
\begin{equation}
\textnormal{inj}_{K}(\rho_{\chi,J})|_I \cong \bigoplus_{\mu=\mu^s} \textnormal{inj}_I(\mu)^{\oplus m_{\rho_{\chi,J}, \mu}}\oplus\bigoplus_{\mu\neq\mu^s} (\textnormal{inj}_I(\mu)\oplus\textnormal{inj}_I(\mu^s))^{\oplus m_{\rho_{\chi,J}, \mu}},
\end{equation}
the sums being taken over $W$-orbits of characters.

We let $\textnormal{P} := \rho_{\chi,J}$, let $\textnormal{X}$ be the representation of $I$ defined by
$$\textnormal{X} := \bigoplus_{\mu=\mu^s} \mu^{\oplus m_{\rho_{\chi,J}, \mu}}\oplus\bigoplus_{\mu\neq\mu^s} (\mu\oplus\mu^s)^{\oplus m_{\rho_{\chi,J}, \mu}},$$
and let $\textnormal{P}'$ be a representation of $K'$ of the form
$$\textnormal{P}' := \bigoplus_{\mu=\mu^s} (\rho'_{\gamma_\mu})^{\oplus m_{\rho_{\chi,J}, \mu}}\oplus\bigoplus_{\mu\neq\mu^s}(\rho'_{\gamma_\mu})^{\oplus m_{\rho_{\chi,J}, \mu}}.$$
Here we choose $\rho'_{\gamma_\mu}\in\{\rho'_{\mu,S'},\rho'_{\mu,\emptyset}\}$ if $\mu=\mu^s$ and $\rho'_{\gamma_\mu}\in\{\rho'_{\mu,\emptyset},\rho'_{\mu^s,\emptyset}\}$ if $\mu\neq\mu^s$; the only stipulation we make is that $\rho_{\chi,J'}'$ be among the summands.  By definition and Corollary \ref{injdecomp}, we have $\textnormal{inj}_K(\textnormal{P})|_I \cong \textnormal{inj}_{K'}(\textnormal{P}')|_I \cong \textnormal{inj}_I(\textnormal{X})$.  It is now evident how to define 
$$\psi: D_{\chi,\tj}\rightarrow E_{\chi,\tj},$$
and that $E_{\chi,\tj}$ is pure with respect to $M_{\chi,\tj}$.  
\end{proof}

\begin{cor}\label{Karol*}
Assume $q = p$, let $M_{\chi, \tj}$ be a supersingular module, and let $E_{\chi, \tj}$ be a pure diagram with respect to $M_{\chi, \tj}$, constructed as in the proof of the previous theorem.  Then the image
\begin{center}
$\pi_{E_{\chi, \tj}}=\textnormal{im}(\psi_*:H_0(X, \mathcal{D}_{\chi, \tj})\rightarrow H_0(X, \mathcal{C}(E_{\chi, \tj}))~)$
\end{center}
is irreducible, admissible, and supersingular.  Moreover, we have $\textnormal{soc}_K(\pi_{E_{\chi,\tj}}|_K)\cong \rho_{\chi,J}$, and for distinct modules $M_{\chi, \tj}, M_{\chi', \tj'}$, the representations $\pi_{E_{\chi, \tj}}, \pi_{E_{\chi', \tj'}}$ are nonisomorphic.   
\end{cor}

\begin{proof}
 The first part of the Corollary follows from Theorems \ref{Karol-3} and \ref{Karol}.  The statement about the $K$-socle follows from the proof of Theorem \ref{Karol-3}.  To prove the last part, let us assume $\phi:\pi_{E_{\chi, \tj}}\stackrel{\sim}{\rightarrow} \pi_{E_{\chi', \tj'}}$ is an isomorphism;  we then obtain an induced isomorphism $${\phi}:\textnormal{soc}_K(\pi_{E_{\chi, \tj}}|_K)^{I(1)}\stackrel{\sim}{\rightarrow} \textnormal{soc}_K(\pi_{E_{\chi', \tj'}}|_K)^{I(1)}.$$  The proof of Theorem \ref{Karol-3} shows how to equip these spaces with an action of $\hh_{\ol{\mathbb{F}}_p}(G,I(1))$, which gives $$M_{\chi, \tj}\cong \textnormal{soc}_K(\pi_{E_{\chi, \tj}}|_K)^{I(1)}\stackrel{{\phi}}{\rightarrow} \textnormal{soc}_K(\pi_{E_{\chi', \tj'}}|_K)^{I(1)}\cong M_{\chi', \tj'}.$$  The claim now follows from the comments following Definition \ref{ssingdef}.
\end{proof}

\begin{remark}
 Assume $q = p$.  Given a supersingular module $M_{\chi, \tj}$, our construction shows that there may be many choices of pure diagram $E_{\chi, \tj}$ associated to $M_{\chi, \tj}$.  As a consequence, if $E_{\chi, \tj}^1$ and $E_{\chi, \tj}^2$ are two such diagrams, we obtain two supersingular representations $\pi_{E_{\chi, \tj}^1}$ and $\pi_{E_{\chi, \tj}^2}$ whose $I(1)$-invariants contain $M_{\chi, \tj}$.  It is not clear, however, if these representations are isomorphic.  
\end{remark}

\section{Some Remarks}

\subsection{The Case $q \neq p$}  In this section we point out the shortcomings of our method in the case when $q \neq p$.  We assume that $q = p^2$ for the sake of simplicity.  

Let $1$ denote the trivial character of $H$ (or, equivalently, of $I$), and consider the diagram $D_{1, (\emptyset, S')} = (\rho_{1,\emptyset},~ \rho'_{1,S'},~ 1,~ j,~ j')$.  Here $\rho_{1,\emptyset}$ is the Steinberg representation of $K$, and $\rho'_{1,S'}$ is the trivial character of $K'$.  We claim that there does not exist a pure diagram $D$ with respect to $M_{1, (\emptyset, S')}$ of the form $(\textnormal{inj}_K(\textnormal{P}),~ \textnormal{inj}_{K'}(\textnormal{P}'),~ \textnormal{inj}_I(\textnormal{X}),~ \boldsymbol{j},~ \boldsymbol{j}')$, where $\textnormal{P}$ is a semisimple representation of $K$, $\textnormal{P}'$ is a semisimple representation of $K'$, and $\textnormal{X}$ is a semisimple representation of $I$.

We require some preparatory facts.  Let $\mu$ and $\mu^\star$ be two $\ol{\mathbb{F}}_p$-characters of $H$ defined by $$\mu\begin{pmatrix}a & 0 & 0 \\ 0 & \delta & 0 \\ 0 & 0 & \ol{a}^{-1}\end{pmatrix} = a^{(p^2 + 1)(p - 1)},\quad \mu^\star\begin{pmatrix}a & 0 & 0 \\ 0 & \delta & 0 \\ 0 & 0 & \ol{a}^{-1}\end{pmatrix} = a^{(p^2 + 1)(p + 1)}.$$
Using the character tables computed in \cite{En63} along with a slightly modified version of Proposition 1.1 in \cite{Di07}, we obtain
$$\textnormal{ind}_{\bb'}^{\Gamma'}(\mu)^{\textnormal{ss}}  \cong  \rho'_{1,S'}~\oplus~ \rho'_{\mu,\emptyset}~\oplus~ \rho'_{\mu^s,\emptyset}~\oplus~ \rho'_{\mu^\star,\emptyset},$$
 where the superscript ``ss'' denotes semisimplification.  Since $\textnormal{SU}(1,1)(\mathbb{F}_{p^4}/\mathbb{F}_{p^2})$ is conjugate to $\textnormal{SL}_2(\mathbb{F}_{p^2})$, we may modify the arguments in Section 4.2 of \cite{Pas04} to show that $\dim_{\ol{\mathbb{F}}_p}(\textnormal{inj}_{\Gamma'}(\rho'_{1,S'})) = 3p^2$.  Combining these two facts with Lemma \ref{Karol-7} shows that 
 \begin{equation}\label{injtriv}
 \textnormal{inj}_{K'}(\rho'_{1,S'})|_I \cong \textnormal{inj}_I(1)\oplus\textnormal{inj}_I(\mu)\oplus\textnormal{inj}_I(\mu^s).
 \end{equation}
 Additionally, $\rho_{1,\emptyset}$ is injective as a representation of $\Gamma$, and therefore
 \begin{equation}\label{injst}
 \textnormal{inj}_K(\rho_{1,\emptyset})|_I \cong \textnormal{inj}_I(1).
 \end{equation}

Assume now that we have an embedding of diagrams $D_{1, (\emptyset, S')}\rightarrow D$, with $D$ pure:

\centerline{
\xymatrix{
 & \rho_{1,\emptyset} \ar[rr]& & \textnormal{inj}_K(\textnormal{P})\\
1\ar[ur]^j\ar[dr]_{j'} \ar[rr]& & \textnormal{inj}_{I}(\textnormal{X})\ar[ur]^{\boldsymbol{j}}\ar[dr]_{\boldsymbol{j}'} & \\
 & \rho'_{1,S'} \ar[rr] & & \textnormal{inj}_{K'}(\textnormal{P}')
}
}
\noindent This implies in particular that $\rho_{1,\emptyset}$ is a direct summand of $\textnormal{P}$ (resp. $\rho_{1,S'}'$ is a direct summand of $\textnormal{P}'$), and hence $\textnormal{inj}_K(\rho_{1,\emptyset})$ is a direct summand of $\textnormal{inj}_K(\textnormal{P})$ (resp. $\textnormal{inj}_{K'}(\rho_{1,S'}')$ is a direct summand of $\textnormal{inj}_{K'}(\textnormal{P}')$).  

Assume first that the $K$-representation of $D$ has simple $K$-socle, so that $\textnormal{P} \cong \rho_{1,\emptyset}$.  The definition of purity gives  
$$\textnormal{inj}_I(1)\oplus\textnormal{inj}_I(\mu)\oplus\textnormal{inj}_I(\mu^s) \stackrel{\textnormal{eq.}~\eqref{injtriv}}{\cong} \textnormal{inj}_{K'}(\rho'_{1,S'})|_I \hookrightarrow \textnormal{inj}_{K'}(\textnormal{P}')|_I\cong\textnormal{inj}_K(\textnormal{P})|_I\stackrel{\textnormal{eq.}~\eqref{injst}}{\cong} \textnormal{inj}_I(1),$$ which is absurd.  

We may therefore assume that the $K'$-representation of $D$ has simple $K'$-socle, so that $\textnormal{P}' \cong \rho'_{1,S'}$ and $$\textnormal{inj}_K(\textnormal{P})|_I\cong \textnormal{inj}_{K'}(\textnormal{P}')|_I \stackrel{\textnormal{eq.}~\eqref{injtriv}}{\cong} \textnormal{inj}_I(1)\oplus\textnormal{inj}_I(\mu)\oplus\textnormal{inj}_I(\mu^s),$$ by the definition of purity.  As $\textnormal{inj}_K(\rho_{1,\emptyset})$ is a summand of $\textnormal{inj}_K(\textnormal{P})$, we have
$$\textnormal{inj}_K(\textnormal{P}/\rho_{1,\emptyset})|_I\cong\textnormal{inj}_K(\textnormal{P})/\textnormal{inj}_K(\rho_{1,\emptyset})|_I \cong \textnormal{inj}_I(\mu)\oplus\textnormal{inj}_I(\mu^s);$$ 
since $m_{\rho_{\chi,J},\chi}\geq 1$, the only representations for which this could potentially be true are $\rho_{\mu,\emptyset}$ and $\rho_{\mu^s,\emptyset}$.  The dimensions of the injective envelopes of $\textnormal{SU}(2,1)(\mathbb{F}_{p^4}/\mathbb{F}_{p^2})$ have been computed explicitly by Dordowsky in his Diplomarbeit (\cite{Do88}).  In particular, his results show that $\dim_{\ol{\mathbb{F}}_p}(\textnormal{inj}_\Gamma(\rho_{\mu,\emptyset})) = \dim_{\ol{\mathbb{F}}_p}(\textnormal{inj}_\Gamma(\rho_{\mu^s,\emptyset})) = 12p^6$, which implies that the number 
of summands in the decompositions of $\textnormal{inj}_K(\rho_{\mu,\emptyset})|_I$ and $\textnormal{inj}_K(\rho_{\mu^s,\emptyset})|_I$ is  12.  This verifies our claim.

\subsection{Comparison with $\textnormal{SL}_2(F)$}
In the course of defining diagrams and coefficient systems for $\textrm{U}(2,1)(E/F)$, there are several parallels one can draw between the formalism we have used and the analogous formalism for the group $\textnormal{SL}_2(F)$.  We hope to make this connection precise here, drawing on results of Abdellatif in \cite{Ab13}.  In this section only, the prime $p$ may be arbitrary.  

We let $G_S := \textnormal{SL}_2(F),~ K_S := \textnormal{SL}_2(\mathfrak{o}_F),$ and $K_S' := \alpha_SK_S\alpha_S^{-1}$, where 
$$\alpha_S := \begin{pmatrix}1 & 0 \\ 0 & \varpi_F\end{pmatrix}.$$
Let $I_S := K_S\cap K_S'$ be the Iwahori subgroup, and $I_S(1)$ its unique pro-$p$-Sylow subgroup.  Let 
$$w_s := \begin{pmatrix}0 & -1 \\ 1 & 0\end{pmatrix}~~\textnormal{and}~~ w_{s'} := \begin{pmatrix}0 & -\varpi_F^{-1}\\ \varpi_F & 0\end{pmatrix},$$ 
and for $r\in \mathbb{Z}$, let $\omega^r$ denote the $\ol{\mathbb{F}}_p$-character of the finite torus $H_S := I_S/I_S(1)$ defined by $$\omega^r\begin{pmatrix}a & 0 \\ 0 & a^{-1}\end{pmatrix} = a^r,$$ where $a\in \mathbb{F}_q^\times$.  

As in Section \ref{halgs}, we denote by $\hh_{\ol{\mathbb{F}}_p}(G_S,I_S(1)) := \textnormal{End}_{G_S}(\textnormal{c-ind}_{I_S(1)}^{G_S}(1))$ the pro-$p$-Iwahori-Hecke algebra, and let $\T_{w_s}$ (resp. $\T_{w_{s'}}$) be the endomorphism corresponding by adjunction to the characteristic function of $I_S(1)w_sI_S(1)$ (resp. $I_S(1)w_{s'}I_S(1)$).  For $0\leq r < q-1$, we define $$e_{\omega^r} := |H_S|^{-1}\sum_{h\in H_S}\omega^r(h)\T_{h},$$ where $\T_h$ corresponds to the characteristic function of $I_S(1)hI_S(1)$.  By Theorem 1 of \cite{Vig05}, $\hh_{\ol{\mathbb{F}}_p}(G_S,I_S(1))$ is generated by $\T_{w_s}, \T_{w_{s'}}$ and $e_{\omega^r}$ for $0\leq r<q-1$.

The supersingular Hecke modules (as defined in \cite{Vig05}) have been classified in \cite{Ab14b}:

\begin{prop}
The supersingular $\hh_{\ol{\mathbb{F}}_p}(G_S,I_S(1))$-modules are all one-dimensional.  They are given by:
 \begin{center}
  \begin{tabular}{ccrclccclcccl}
   $M_{0}$ & : & $e_{1}$ & $\mapsto$ & $1$, & & $\T_{w_s}$ & $\mapsto$ & $\phantom{-}0$, & & $\T_{w_{s'}}$ & $\mapsto$ & $-1$;\\
   $M_{q-1}$ & : & $e_1$ & $\mapsto$ & $1$, & & $\T_{w_s}$ & $\mapsto$ & $-1$, & & $\T_{w_{s'}}$ & $\mapsto$ & $\phantom{-}0$;\\
   $M_{r}$ & : & $e_{\omega^r}$ & $\mapsto$ & $1$, & & $\T_{w_s}$ & $\mapsto$ & $\phantom{-}0$, & & $\T_{w_{s'}}$ & $\mapsto$ & $\phantom{-}0$,
  \end{tabular}
 \end{center}
where $0 < r < q-1$.  
\end{prop}

As is the case for $\textrm{U}(2,1)(E/F)$, the Bruhat-Tits building $X_S$ of $G_S$ is a tree.  The action of $G_S$ partitions the vertices into two orbits (those at an even (resp. odd) distance from the vertex corresponding to $K_S$), and is transitive on the set of (nonoriented) edges.  Hence, the notion of a diagram is the same as in Definition \ref{diag}, and the results of Section \ref{diagsandcoeffs} carry over formally for the group $G_S$.  In particular, the categories $\mathcal{COEF}_{G_S}$ and $\mathcal{DIAG}$ are equivalent.  With this analogy in mind, we define the following diagrams.  

\begin{defn}\label{sl2diags}  
Let $\textnormal{Sym}^r(\ol{\mathbb{F}}_p^2)$ denote the $r^{\textnormal{th}}$ symmetric power of the standard representation of $\textnormal{SL}_2(\mathbb{F}_q)$, viewed as a representation of $K_S$ and $K_S'$ by inflation.  We define:
\begin{center}
  \begin{tabular}{ccl}
   $D_{0}$ & := & $\left(\textnormal{Sym}^0(\ol{\mathbb{F}}_p^2),\quad\textnormal{Sym}^{p-1}(\ol{\mathbb{F}}_p^2)\otimes\textnormal{Sym}^{p-1}(\ol{\mathbb{F}}_p^2)^{\textnormal{Fr}}\otimes\cdots\otimes\textnormal{Sym}^{p-1}(\ol{\mathbb{F}}_p^2)^{\textnormal{Fr}^{f-1}},~ 1,~ j,~ j'\right)$;\\
   $D_{q-1}$ & := & $\left(\textnormal{Sym}^{p-1}(\ol{\mathbb{F}}_p^2)\otimes\textnormal{Sym}^{p-1}(\ol{\mathbb{F}}_p^2)^{\textnormal{Fr}}\otimes\cdots\otimes\textnormal{Sym}^{p-1}(\ol{\mathbb{F}}_p^2)^{\textnormal{Fr}^{f-1}},\quad\textnormal{Sym}^{0}(\ol{\mathbb{F}}_p^2),~ 1,~ j,~ j'\right)$;\\
   $D_{r}$ & := & $\left(\textnormal{Sym}^{r_0}(\ol{\mathbb{F}}_p^2)\otimes\textnormal{Sym}^{r_1}(\ol{\mathbb{F}}_p^2)^{\textnormal{Fr}}\otimes\cdots\otimes\textnormal{Sym}^{r_{f - 1}}(\ol{\mathbb{F}}_p^2)^{\textnormal{Fr}^{f-1}},\right.$ \\
   & & $\qquad\left.  \textnormal{Sym}^{p - 1 - r_0}(\ol{\mathbb{F}}_p^2)\otimes\textnormal{Sym}^{p - 1 - r_1}(\ol{\mathbb{F}}_p^2)^{\textnormal{Fr}}\otimes\cdots\otimes\textnormal{Sym}^{p - 1 - r_{f - 1}}(\ol{\mathbb{F}}_p^2)^{\textnormal{Fr}^{f-1}},~ \omega^{r},~ j,~ j'\right)$,
  \end{tabular}
 \end{center}
where $0 < r = \sum_{i = 0}^{f - 1} r_ip^i< q-1$ is the $p$-adic expansion of $r$, and where $j$ and $j'$ are inclusion maps.  
\end{defn}

Using the same arguments as in Section \ref{init}, one can show that given a diagram $D_r$ of the form above, the $I_S(1)$-invariants of every nonzero irreducible quotient of $H_0(X_S, \mathcal{C}(D_r))$ contain $M_r$, and therefore such a quotient must be supersingular.  Specializing to the case $q = p$, we obtain the following:

\begin{prop}\label{pureforsl2}
  Assume $q = p$.  Let $0 \leq r \leq p - 1$, let $M_r$ a supersingular $\hh_{\ol{\mathbb{F}}_p}(G_S,I_S(1))$-module, and let 
  $$D_r = (\textnormal{Sym}^r(\ol{\mathbb{F}}_p^2),~ \textnormal{Sym}^{p - 1 - r}(\ol{\mathbb{F}}_p^2),~ \omega^r,~ j,~ j')$$ 
  be the associated diagram as in Definition \ref{sl2diags}.  Then the diagram 
  $$E_r = (\textnormal{inj}_{K_S}(\textnormal{Sym}^r(\ol{\mathbb{F}}_p^2)),~ \textnormal{inj}_{K_S'}(\textnormal{Sym}^{p - 1 - r}(\ol{\mathbb{F}}_p^2)),~ \textnormal{inj}_{K_S}(\textnormal{Sym}^r(\ol{\mathbb{F}}_p^2))|_{I_S},~ \boldsymbol{j},~ \boldsymbol{j}')$$ 
  is pure with respect to $M_r$, where $\boldsymbol{j}$ and $\boldsymbol{j}'$ are isomorphisms.  
\end{prop}

\begin{thm}\label{ssingsl2}
 Assume $q = p$.  Let $0\leq r \leq p - 1$, let $M_r, D_r,$ and $E_r$ be as in Proposition \ref{pureforsl2}, and let $\psi:D_r\rightarrow E_r$ denote the natural embedding.  Then the representation afforded by $$\textnormal{im}(\psi_*:H_0(X_S,\mathcal{C}(D_r))\rightarrow H_0(X_S,\mathcal{C}(E_r))~)$$ is irreducible, admissible and supersingular.  For distinct supersingular modules $M_r, M_{r'}$, the resulting representations are nonisomorphic.  
\end{thm}

\begin{proof}
 The proof is identical to the proofs of Propositions \ref{reduced to the irr quotient of H_0} and \ref{Karol-3}, and Corollary \ref{Karol*}.
\end{proof}

In this way, we have constructed $p$ irreducible supersingular representations, corresponding to the supersingular $\hh_{\ol{\mathbb{F}}_p}(G_S,I_S(1))$-modules.  In particular, for $F = \mathbb{Q}_p$, we recover the following classification of supersingular representations:

\begin{thm}\label{SL2}
 Let $M_r$ be a supersingular Hecke module for $\textnormal{SL}_2(\mathbb{Q}_p)$, and let $D_r$ and $E_r$ be the diagrams constructed above.  We then have
$$\textnormal{im}(\psi_*:H_0(X_S,\mathcal{C}(D_r))\rightarrow H_0(X_S,\mathcal{C}(E_r))~)\cong \pi_r,$$
where $\pi_r$ is the supersingular representation of $\textnormal{SL}_2(\mathbb{Q}_p)$ defined in \cite{Ab13}.  
\end{thm}

\begin{proof}
By Th\'{e}or\`{e}me 4.12 of \emph{loc. cit.} and Theorem \ref{ssingsl2}, we must have
$$\textnormal{im}(\psi_*:H_0(X_S,\mathcal{C}(D_r))\rightarrow H_0(X_S,\mathcal{C}(E_r))~) \cong \pi_t$$
for some $0\leq t \leq p - 1$.  By Propositions 7.4, 7.6, and 7.7 of \cite{Ab14b} and the discussion above, we also have
 $$M_r\subset\textnormal{im}(\psi_*:H_0(X_S,\mathcal{C}(D_r))\rightarrow H_0(X_S,\mathcal{C}(E_r))~)^{I_S(1)} \cong \pi_t^{I_S(1)}\cong M_t,$$
and therefore $t = r$.  
\end{proof}

\begin{remark}
 When $q \neq p$, the above construction fails in a manner similar to the construction for $\textrm{U}(2,1)(E/F)$, meaning that pure diagrams of the form $(\textnormal{inj}_{K_S}(\textnormal{P}),~ \textnormal{inj}_{K_S'}(\textnormal{P}'),~ \textnormal{inj}_{I_S}(\textnormal{X}),~ \boldsymbol{j},~ \boldsymbol{j}')$ do not always exist.  One may translate the example of the previous section to the case of $\textnormal{SL}_2(F)$ to produce such an example explicitly.  
\end{remark}

\appendix

\section{Proof of Theorem \ref{ps}}\label{app}
Here we carry out the computations for Theorem \ref{ps}.  Recall that $\widetilde{\chi} = \widetilde{\zeta}\otimes\widetilde{\eta}$ is a character of the torus $T$, and the space of $I(1)$-invariants $\textnormal{ind}_B^G(\widetilde{\chi})^{I(1)}$ is spanned by the functions $f_1$ and $f_2$ (cf. Section \ref{ssingmodssec}).  The computations are split up according to the nature of $\chi$.  

Using Proposition 6 in \cite{BL94}, we can compute the action of $\T_{n_s}$ and $\T_{n_{s'}}$ on $f_i \in \textrm{ind}_B^G(\widetilde{\chi})^{I(1)}$.  Equation \eqref{decs} implies
$$f_i\cdot \T_{n_s} = \sum_{\sub{x,y\in \mathbb{F}_{q^2}}{x\ol{x} + y + \ol{y} = 0}}u(-x,\ol{y})n_s^{-1}.f_i.$$
Evaluating at $1$ and $n_s$ gives
\begin{eqnarray*}
 f_i\cdot \T_{n_s}(1) & = & \sum_{\sub{x,y\in \mathbb{F}_{q^2}}{x\ol{x} + y + \ol{y} = 0}}f_i(n_s^{-1})\\
 & = & 0\\
 f_i\cdot \T_{n_s}(n_s) & = & \sum_{\sub{x,y\in \mathbb{F}_{q^2}}{x\ol{x} + y + \ol{y} = 0}}f_i(n_su(-x,\ol{y})n_s^{-1})\\
 & = & \sum_{\sub{x,y\in \mathbb{F}_{q^2}}{x\ol{x} + y + \ol{y} = 0}}f_i(u^-(-\ol{x}\sqrt{\epsilon},-\ol{y}\epsilon))\\
 & \stackrel{\textnormal{eq.}~ \eqref{inv}}{=}& f_i(1) + \sum_{\sub{x,y\in \mathbb{F}_{q^2},y\neq 0}{x\ol{x} + y + \ol{y} = 0}}f_i(h_s(y^{-1}\sqrt{\epsilon}^{-1})n_s)\\
 & = & f_i(1) + f_i(n_s)\left(\sum_{\sub{x,y\in \mathbb{F}_{q^2},y\neq 0}{x\ol{x} + y + \ol{y} = 0}}\widetilde{\chi}(h_s(y^{-1}\sqrt{\epsilon}^{-1}))\right)\\
 & = & f_i(1) + f_i(n_s)\cdot\begin{cases}-1 & \textrm{if}~ \chi~\textnormal{is of trivial type,} \\ \phantom{-}0 & \textrm{if}~\chi~\textnormal{is hybrid},\\ \phantom{-}0 & \textrm{if}~ \chi~\textnormal{is regular.}\end{cases}
\end{eqnarray*}
where the last equality is obtained in precisely the same manner as in the proof of Proposition \ref{str} (cf. the computation of $\mT_{n_s}^2$).  

Likewise, equation \eqref{decs'} implies 
$$f_i\cdot \T_{n_{s'}} = \sum_{\sub{y\in \mathbb{F}_{q^2}}{y + \ol{y} = 0}}u^-(0,\varpi\ol{y})\alpha n_s^{-1}.f_i,$$ 
and thus we have
\begin{eqnarray*}
 f_i\cdot \T_{n_{s'}}(1) & = & \sum_{\sub{y\in \mathbb{F}_{q^2}}{y + \ol{y} = 0}}f_i(u^-(0,\varpi\ol{y})\alpha n_s^{-1})\\
 & \stackrel{\textnormal{eq.}~\eqref{inv}}{=} & \widetilde{\zeta}(-1)\widetilde{\chi}(\alpha)f_i(n_s) + \sum_{\sub{y\in \mathbb{F}_{q^2}, y\neq 0}{y + \ol{y} = 0}}f_i(h_s(-\varpi^{-1}y^{-1}\sqrt{\epsilon})n_su(0,\varpi^{-1}y^{-1})\alpha n_s^{-1})\\
 & = & \widetilde{\zeta}(-1)\widetilde{\chi}(\alpha)f_i(n_s) + f_i(1)\left(\sum_{\sub{y\in \mathbb{F}_{q^2}, y\neq 0}{y + \ol{y} = 0}}\widetilde{\chi}(h_s(-y^{-1}\sqrt{\epsilon}))\right)\\
 & = & \widetilde{\zeta}(-1)\widetilde{\chi}(\alpha)f_i(n_s) + f_i(1)\cdot \begin{cases}-1 & \textrm{if}~ \chi~\textnormal{is of trivial type,}\\ -1 & \textrm{if}~ \chi~\textnormal{is hybrid,}\\ \phantom{-}0 & \textrm{if}~ \chi~\textnormal{is regular.}\end{cases}\\
f_i\cdot \T_{n_{s'}}(n_s) & = & \sum_{\sub{y\in \mathbb{F}_{q^2}}{y + \ol{y} = 0}}f_i(n_su^-(0,\varpi\ol{y})\alpha n_s^{-1})\\
&  = & \sum_{\sub{y\in \mathbb{F}_{q^2}}{y + \ol{y} = 0}}\widetilde{\chi}(\alpha^{-1})f_i(1)\\
& = & 0.
\end{eqnarray*}

\end{document}